\pdfoutput=1
\def\submission{0}
\def\com{0}
\def\supplementary{0}
\ifnum\submission=1
\documentclass{article}
\usepackage{times}
\else
\documentclass[11pt]{article}
\usepackage[affil-it]{authblk}
\fi

\usepackage{algorithm}
\ifnum\submission=1
\usepackage[noend]{algorithmic}

\newcommand{\LINEIF}[2]{%
    \STATE\algorithmicif\ {#1}\ \algorithmicthen\ {#2} %
}
\newcommand{\LINEELSE}[1]{%
    \STATE\algorithmicelse\   {#1}%
}
\newcommand{\LINERETURN}[1]{%
    \STATE\algorithmicreturn\   {#1}%
}
\newcommand{\return}{\textbf{return}}

\usepackage{graphicx} % more modern
\usepackage{subfigure} 
\else
\usepackage[noend]{algpseudocode}
\usepackage{caption}
\usepackage{subcaption}
\fi
\usepackage{fullpage}
\usepackage{amsmath, amssymb, amsthm}
\usepackage{verbatim}
\usepackage{color}
\usepackage{bbm} 
\usepackage{xcolor}
\definecolor{DarkGreen}{rgb}{0.1,0.5,0.1}
\definecolor{DarkRed}{rgb}{0.5,0.1,0.1}
\definecolor{DarkBlue}{rgb}{0.1,0.1,0.5}
\ifnum\submission=1
\usepackage[accepted]{icml2016}
\fi

\usepackage[nottoc,notlot,notlof]{tocbibind}

\usepackage{natbib}
\usepackage{tikz}
\usepackage{tikz-qtree}

\ifnum\com=1
\newcommand{\rynote}[1]{\textcolor{blue}{[Ryan: #1]}}
\newcommand{\mnote}[1]{\textcolor{red}{[Marco: #1]}}
\else
\newcommand{\rynote}[1]{}
\newcommand{\mnote}[1]{}
\fi

\ifnum\submission=1
  \newcommand{\mathmath}{$ }
 \else
   \newcommand{\mathmath}{$$ }
 \fi

\usepackage[bookmarks=false]{hyperref}
\usepackage{cleveref}

\newcommand\N{\mathbb{N}}
\newcommand\R{\mathbb{R}}
\newcommand{\cA}{\mathcal{A}}

\newcommand{\cD}{\mathcal{D}}

\DeclareMathOperator*{\myargmin}{\arg\!\min}

\newcommand{\bbf}{\mathbf{f}}

\newcommand{\bp}{\mathbf{p}}

\newcommand{\bbx}{\mathbf{x}}

\newcommand{\bV}{\mathbf{V}}
\newcommand{\bbb}{\mathbf{b}}

\newcommand{\bbt}{\mathbf{t}}

\newcommand{\bd}{\mathbf{d}}

\newcommand{\bX}{\mathbf{X}}
\newcommand{\bZ}{\mathbf{Z}}

\newcommand{\sigeps}{\sigma(\epsilon,\delta_n)}
\newcommand{\NULL}{\mathrm{NULL}}

\newcommand{\Bern}{\mathrm{Bern}}
\newcommand{\Cov}{\mathrm{Cov}}
\newcommand{\Var}{\mathrm{Var}}

\newcommand\poly{\mathrm{poly}}

\renewcommand{\hat}{\widehat}
\renewcommand{\tilde}{\widetilde}

\newcommand{\ex}[1]{\mathbb{E}\left[#1\right]}
\DeclareMathOperator*{\Expectation}{\mathbb{E}}
\newcommand{\Ex}[2]{\Expectation_{#1}\left[#2\right]}
\DeclareMathOperator*{\Probability}{\mathrm{Pr}}

\newcommand{\Prob}[2]{\Probability_{#1}\left[#2\right]}

\newcommand{\Lap}{\mathrm{Lap}}
\renewcommand{\hat}{\widehat}

\newcommand{\MC}{\texttt{MCGOF}}
\newcommand{\PG}{\texttt{PrivGOF}}
\newcommand{\GOF}{\texttt{GOF}}
\newcommand{\PI}{\texttt{PrivIndep}}

\newcommand{\MLE}{\texttt{2MLE}}

\newcommand{\IND}{\texttt{Indep}}
\newcommand{\MCIND}{\texttt{MCIndep}_\cD}
\newcommand{\YY}[1]{\mathbf{Y}^{(#1)}}
\newcommand{\pipi}[1]{\pi^{(#1)}}
\newcommand{\tpipi}[1]{\tilde{\pi}^{(#1)}}
\newcommand{\ttpipi}[1]{\tilde{\tilde{\pi}}^{#1}}

\newcommand{\mult}{\mathrm{Multinomial}}
\newcommand{\diag}{\mathrm{Diag}}

\newtheorem{theorem}{Theorem}[section]
\newtheorem{lemma}[theorem]{Lemma}

\newtheorem{corollary}[theorem]{Corollary}

\theoremstyle{definition}
\newtheorem{definition}[theorem]{Definition}

\theoremstyle{remark}
\newtheorem{remark}[theorem]{Remark}

\ifnum\submission=1
\icmltitlerunning{Differentially Private Chi-Squared Hypothesis Testing}
\fi

\begin{document}

\ifnum\submission=1
\twocolumn[
\icmltitle{Differentially Private Chi-Squared Hypothesis Testing: \\ Goodness of Fit and Independence Testing}
\icmlauthor{Marco Gaboardi}{gaboardi@buffalo.edu}
\icmladdress{University at Buffalo, SUNY}
\icmlauthor{Hyun Woo Lim}{limhyun@g.ucla.edu}
\icmladdress{University of California, Los Angeles}
\icmlauthor{Ryan Rogers}{ryrogers@sas.upenn.edu}
\icmladdress{University of Pennsylvania}
\icmlauthor{Salil P. Vadhan}{salil@seas.harvard.edu}
\icmladdress{Harvard University}

\icmlkeywords{boring formatting information, machine learning, ICML}
\vskip 0.3in
]
\else
	\ifnum\supplementary=1
	\title{Differentially Private Chi-Squared Hypothesis Testing: \\ Goodness of Fit and Independence Testing}
	\author{SUPPLEMENTARY MATERIAL - FULL VERSION}
	\else
	\title{Differentially Private Chi-Squared Hypothesis Testing: \\ Goodness of Fit and Independence Testing\thanks{This work is part of the ``Privacy Tools for Sharing Research Data'' project based at Harvard, supported by NSF grant CNS-1237235 as well as a grant from the Sloan Foundation.}}
	\author{Marco Gaboardi\thanks{This work has been partially supported by the ``PrivInfer - Programming Languages for Differential Privacy: Conditioning and Inference'' EPSRC project EP/M022358/1 and by the University of Dundee, UK.}}
	\affil{University at Buffalo, SUNY}
	\author{Hyun woo Lim}
	\affil{University of California, Los Angeles}
	\author{Ryan Rogers}
	\affil{University of Pennsylvania}
	\author{Salil P. Vadhan\thanks{Also supported by a Simons Investigator grant.  Work done in part while visiting the Department of Applied Mathematics and the Shing-Tung Yau Center at National Chiao-Tung University in Taiwan.}}
	\affil{Harvard University}
	\fi
\maketitle
\fi

\begin{abstract}
  Hypothesis testing is a useful statistical tool in determining whether a given model should be rejected based on a sample from the population. Sample data may contain sensitive information about individuals, such as medical information. Thus it is important to design statistical tests that guarantee the privacy of subjects in the data.  In this work, we study hypothesis testing subject to differential privacy, specifically chi-squared tests for goodness of fit for multinomial data and independence between two categorical variables. 

\ifnum\submission=0
 We propose new tests for goodness of fit and independence testing that like the classical versions can be used to determine whether a given model should be rejected or not, and that additionally can ensure differential privacy. We give both Monte Carlo based hypothesis tests as well as hypothesis tests that more closely follow the classical chi-squared goodness of fit test and the Pearson chi-squared test for independence.  Crucially, our tests account for the distribution of the noise that is injected to ensure privacy in determining significance.

 We show that these tests can be used to achieve desired significance levels, in sharp contrast to direct applications of classical tests to differentially private contingency tables which can result in wildly varying significance levels. Moreover, we study the statistical power of these tests.  We empirically show that to achieve the same level of power as the classical non-private tests our new tests need only a relatively modest increase in sample size.
  \fi
\end{abstract}
\ifnum\submission=0
\newpage
 \tableofcontents
 \newpage
 \fi
\section{Introduction}
%Brief introduction to Hypothesis testing and data analysis

%A common task for a data analyst is to test whether a given population model should be rejected based on a sample from that same population.  
Hypothesis testing provides a systematic way to test given models based on a sample, so that with high confidence a data analyst may conclude that the model is incorrect or not.  However, these data samples may contain highly sensitive information about the subjects and so the privacy of individuals can be compromised when the results of a data analysis are released.   %This may happen also when the data has been carefully sanitized. 
For example, in the area of  \emph{genome-wide
   associaton studies} (GWAS) \citet{Homer08} have shown that it is possible to identify subjects in a data set based on publicly available aggregate statistics.

A way to address this concern is by
 developing new techniques to support privacy-preserving data
 analysis. An approach that is gaining more and more attention by the statistics
 and data analysis community is \emph{differential
   privacy}~\citep{DMNS06}, which originated in theoretical computer science.  In this work, we seek to develop hypothesis tests 
that are differentially private and that give conclusions similar to standard, non-private hypothesis tests.  
% that are both differentially private and would have similar conclusions as if we carried out a hypothesis test non-privately.  
%   A standard way to satisfy differential privacy is by carefully adding some random noise to the data itself or to an aggregate statistic of the data analysis. It is important to know the scale of the noise that is required to ensure differential privacy while simultaneously ensuring that the utility of the statistical analysis is not ruined by the additional noise.
%%%%%%
%The research on designing data analyses that ensures differential
%privacy is very active and has generated many interesting examples
%that ensure privacy while providing good accuracy,
%for instance~\citep{DMNS06,MT07,GLMRT10,BLR13,CHS14}; see also \citet{PrivBook} for a textbook exposition of differential privacy and its applications.

%There are several important hypothesis test that are frequently used by statistician and data analyst. These varies in the kind of data they can be used on, and in the kind of hypothesis that can actually be tested. 
We focus here on two classical tests for data drawn from a
multinomial distribution: \emph{goodness of fit test}, which determines whether the data was in fact drawn from a multinomial distribution with probability vector $\bp^0$; and \emph{independence test},  which tests whether two categorical random variables are independent of each other.  Both tests depend on the \emph{chi-squared} statistic, which is used to determine whether the data is likely or not under the given model.

%These two tests are based on the assumption that the sample distribution must have the form of a chi-squared distribution if the null hypothesis is true. So, both the tests compute the \emph{chi-squared statistic} of their input data and check whether the obtained value is above or below a given significance threshold.

% In particular, we develop differentially private
% hypothesis tests based
% on the  chi-squared statistic for categorical data drawn from a
% multinomial distribution.
%The multinomial distribution can be thought of as the generalization of the binomial distribution, where instead of two outcomes (success or failure) there are actually $d$ outcomes where the probability of being in each is given as a probability vector $\bp = (p_1,\cdots, p_d)$.  The random variable $\mathbf{X} \sim Multinomial (n,\bp)$ then is a vector of counts, where each component $X_i$ is the number of $n$ independent trials that resulted in outcome $i$.  
%\mnote{I find the following two sentences misplaced here.} It is common in hypothesis testing to form a statistic based on the data and find its \emph{asymptotic} distribution given a \emph{null hypothesis}.  This asymptotic distribution is then used to determine whether the statistic is likely to have occurred or not.

To guarantee differential privacy, we consider adding Laplace and Gaussian noise to the counts of categorical data. Using the noisy data  we can form a \emph{private} chi-squared statistic. It turns out that the classical hypothesis tests perform poorly when used with this modified statistic because they ignore the fact that noise was added.  
To improve this situation, we develop four new tests that account for the additional noise due to privacy. In particular, we give two tests based on a \emph{Monte Carlo} approach to testing the null hypothesis and two tests based on an \emph{asymptotic} distribution of the private chi-squared distribution factoring in the noise distribution. 

Our four differentially private tests achieve a target level $1-\alpha$ \emph{significance}, i.e. they reject with probability at most $\alpha$ when the null hypothesis holds (in some cases, we provide a rigorous proof of this fact and in others, it is experimentally verified).  This guarantees limited Type I errors. 
However, all of our tests do lose \emph{power}; that is when the null hypothesis is false, they correctly reject with lower probability than the classical hypothesis tests. This corresponds to an increase of Type II errors. 
We empirically show that we can recover a level of power similar to the one achieved by the classical versions by adding more samples.  
%Adding is thus a way to minimize Type II errors under the hard constraint of Type I errors. \mnote{Ryan, check on this please.}

% For reasonably sized data sets, we want our tests to simultaneously be able to have at least a $1-\alpha$ level of \emph{significance}, i.e. not reject when the null hypothesis is in fact true, and \emph{power}, i.e. correctly rejecting the null hypothesis when it is false, that is close to the classical hypothesis tests without privacy.  
% We show that indeed our tests do lose power when compared to the classical versions, but the number of samples needed to recover a certain level of power for our differentially private tests are not substantially more than the number of samples needed without privacy.  For a fixed privacy level, our MC based tests with Laplace noise does not lose as much power as when we use Gaussian noise, due to the fact that the Gaussian noise has higher variance.  However when Gaussian noise is used, we derive the asymptotic distribution for the private chi-squared statistic that can be used to determine p-values or thresholds in which to reject the null hypothesis that is similar to the classical tests without privacy.

\ifnum\submission=0
\subsection{Contributions}

For goodness of fit testing we present two differentially private tests.  First, we give a Monte Carlo (MC) based test $\MC_\cD$ that guarantees significance at least $1-\alpha$ for any desired $\alpha>0$ (commonly $0.05$ or $0.10$), when either Laplace or Gaussian noise is added in the private chi-squared statistic.  
%We also empirically compare the power of the classical goodness of fit test with our MC test for both Laplace and Gaussian noise when an alternate hypothesis $\bp^1$ holds, which is close to the null hypothesis probability vector $\bp^0$.  
When Gaussian noise is used in the private chi-squared statistic and has privacy parameter that decays with the sample size, we then analytically obtain an asymptotic distribution for the private chi-squared statistic that is a linear combination of independent chi-squared random variables with one degree of freedom, given that the null hypothesis is true.  We then use this asymptotic distribution for the private chi-squared statistic for the test $\PG$.  This provides an alternative to our MC test, which is computationally less expensive and more closely parallels the classical test, which is based on the asymptotic distribution of the nonprivate statistic (which has an asymptotic chi-squared distribution with $d-1$ degrees of freedom, where $d$ is the dimension of the data).  Further, when the data actually satisfies an alternate hypothesis $\bp^1 \neq \bp^0$ that has a specific form, we find the asymptotic distribution of the private chi-squared, which is to be compared with the non-private chi-squared statistic converging to a non-central chi-squared distribution with $d-1$ degrees of freedom when the alternate hypothesis is true.  

We then turn to independence testing.  Given a contingency table where each cell has Laplace or Gaussian noise added, inspired by \citet{KS16}, we present a (heuristic) procedure $\MLE$ for finding an approximate maximum likelihood estimator (MLE) for the true probability vector that satisfies the independence hypothesis.  Note that for the classical Pearson chi-squared test, one computes the MLE for the true probability vector given a contingency table (without additional noise).  We then use this estimate as the probability vector in our private chi-squared statistic and give an MC method $\MCIND$ based on the private chi-squared statistic. 

Lastly, when dealing with Gaussian noise, we give a differentially private test $\PI$ that closely follows the analysis of the Pearson chi-squared test for independence.  We show that when we approximate the distribution of the private chi-squared statistic with a linear combination of chi-squared random variables with 1 degree of freedom we obtain empirical significance at least $1-\alpha$ and power that closely follows the power of $\MCIND$ that uses Gaussian noise for various sample sizes $n$.

For all of our tests we give empirical significance and power results and compare them to their corresponding classical, non-private tests.  We obtain empirical significances that are near the desired $1-\alpha$ level and clearly outperform the classical tests when they are used on the noisy data.  In particular, this guarantees that our tests do not incur more Type I error than the classical tests. However, our tests do have lower power than the corresponding non-private test, due to the additional noise injected. Thus, our tests have larger Type II errors than the classical (nonprivate) tests. In particular, the tests using Gaussian noise have a larger  loss in power than the ones using Laplace noise.  This is because for a given level of privacy, the Gaussian noise has a larger variance than the Laplace noise. To achieve power similar to the classical tests, our private tests require more samples. 
\fi

%Our context: the Harvard privacy project
\ifnum\submission=0
\subsection{Hypothesis testing for the social sciences}
Our work is part of the broader effort of the project ``Privacy Tools
for Sharing Research
Data''\footnote{\url{http://privacytools.seas.harvard.edu}} that aims
in particular at developing differentially private tools that can be
used for studies in the social sciences. Social scientists often deal with various sensitive data that contains individual's private information, e.g. voting behavior \citep{ind5}, attitude toward abortion \citep{ind2} and medical records \citep{gof1}.  The framework of hypothesis testing
is frequently used by social scientists to confirm or reject their
belief to how a population is modeled, e.g. goodness of fit tests have been used by~\citep{gof1,gof2,gof3,gof4} and independence
tests have been used by~\citep{ind1,ind2,ind3,ind4,ind5,ind6}.
\fi
 
%%% Local Variables: 
%%% mode: latex
%%% TeX-master: "main"
%%% End: 

  \section{Related Work} 

There has been a myriad of work dealing with the application of differential privacy in statistical inference.  One of the first works that put differential privacy in the language of statistics is \citet{WZ10}, which studies rates of convergence of distributions based on differentially private data released from the \emph{exponential mechanism} \citep{MT07}.  In a result of great generality, \citet{Smith11} shows that for a wide class of statistics $T$, there is a differentially private statistic that converges in distribution to the same asymptotic distribution as $T$.  However, having the correct asymptotic distribution does not ensure that only statistically significant conclusions are drawn at finite sample sizes, and indeed we observe that this fails dramatically for the most natural differentially private algorithms. Thus, we study how to ensure significance and optimize power at small sample sizes by focusing on two basic statistical tests.

A tempting first approach to developing a hypothesis test for categorical data that is also differentially private is to either add  noise  directly to the chi-squared statistic that will ensure differential privacy or to add noise to each cell count (as we do in this work) and use a classical test with the private counts.  For the former method, the amount of noise that must be added to ensure privacy can be unbounded in the worst case.  However, motivated by applications to genome-wide association studies (GWAS), \citet{USF13} and \citet{YFSU14} place restrictions on the form of the data or what is known to the data analyst to reduce the scale of the noise that needs to be added.  The work of \citet{JS13} adds noise to each cell of a contingency table, but then uses classical statistical tests on the private version of the data, which we show can have very poor significance (see \Cref{fig:GOF_sig,fig:IND_sig}).  Additionally, \citet{USF13} look at $3\times 2$ contingency tables that are evenly split between the two columns, and study releasing differentially private $\chi^2$-statistics of the most relevant SNPs for certain diseases by perturbing the table of counts, the $\chi^2$-statistic itself, and the $p$-values for the underlying test.  The only one of these works that explicitly examine significance and power in hypothesis testing (as we do here) is \cite{USF13}, which shows that perturbing the $p$-values in independence testing does not perform much better than a random test, independent of a selected threshold, e.g. $\alpha$.  In fact, \cite{USF13} goes as far as to say that basing inference on perturbed $p$-values ``seems impossible."  An interesting direction for future work would be to apply the \emph{distance-score mechanism} introduced by \cite{JS13} and later improved by \cite{YFSU14, YJ14, SB16}, to achieve a target level of significance and high power in hypothesis testing for GWAS data. 

If we assume that there is some prior estimates for the contingency table cell probabilities, \citet{VS09} determine the sample size adjustment for the Pearson chi-squared independence test that uses the private counts to achieve the same power as the test with the original counts.  Several other works have shown negative experimental results for using classical inference on statistics that have been altered for differential privacy \citep{FRY10,KS12, KS16}.

Another problem that arises when noise is added to the cells in a contingency table is that the entries may neither be positive nor sum to a known value $n$.  Several works have focused on this problem, where they seek to release a contingency table in a differentially private way that also satisfies some known \emph{consistency} properties of the underlying data \citep{BarakCDKMT07,CHRMM10,HLM12,LM12,GGHRWlong14}.  For independence testing, we use techniques from \citet{KLW15} to find the most likely contingency table given the noisy version of it so that we can then estimate the cell probabilities that generated the table.  This two step procedure to estimate parameters given a differentially private statistic is inspired by the work of \citet{KS16} for estimating parameters in the $\beta$-model for random graphs.

Independent of our work, \citet{WLK15} also look at hypothesis testing with categorical data subject to differential privacy.  They mainly consider adding Laplace noise to the data but point out that their method also generalizes to arbitrary noise distributions.  However, in order to compute critical values, they resort to Monte Carlo methods to sample from the asymptotic distribution.  Our Monte Carlo approach samples from the \emph{exact} distribution from the underlying null hypothesis, which, unlike sampling from the asymptotic distribution, guarantees significance at least $1-\alpha$ in goodness of fit tests at finite sample sizes.  We only focus on Gaussian noise in our asymptotic analysis due to there being existing methods for finding tail probabilities (and hence critical values) for the resulting distributions, but our approaches can be generalized for arbitrary noise distributions.  Further, we also consider the power of each of our differentially private tests.

  \section{Differential Privacy Preliminaries}
 We start with a brief overview of differential privacy.  In order to define differential privacy, we first define \emph{neighboring} databases $\bd, \bd'$ from some class of databases $D^n$ where they differ in an individual's data but are equal among the rest of the data, e.g. $\bd= (d_1,\cdots, d_i,\cdots d_n)$ and $\bd' = (d_1,\cdots, d_i',\cdots, d_n)$ where $d_i \neq d_i'$.  We will consider $n$ to be known and public.
\begin{definition}[Differential Privacy \citep{DMNS06}]
Let $M: D^n \to O$ be some randomized mechanism
% that takes a database with $n$ records and gives some output from some outcome space $O$.  
For $\epsilon, \delta>0$ we say that $M$ is $(\epsilon,\delta)$-differentially private if for any neighboring databases $\bd,\bd' \in D^n$ and any subset of outcomes $S \subseteq O$ we have
$$
\Prob{}{M(\bd) \in S} \leq e^\epsilon \Prob{}{M(\bd') \in S} + \delta.
$$
If $\delta = 0$, then we simply say $M$ is $\epsilon$-differentially private.  \ifnum\submission=1 As is common in the privacy literature\else The meaning of these parameters, loosely, is that with probability $1-\delta$ there is at most $\epsilon$ information leakage (so with probability at most $\delta$ it can leak lots of information).  For this reason\fi, we will think of $\epsilon$ as a small constant, e.g. $0.1$, and $\delta \ll 1/n$ as cryptographically small, where we sometimes write $\delta_n$ to explicitly show its dependence on $n$.  
\end{definition}

A typical differentially private mechanism is to add carefully
calibrated noise to some quantity that a data analyst is interested
in.  We can release a differentially private answer to a function $\phi: D^n \to \R^d$ by adding independent noise to each component of $\phi$.  The scale of the noise we add
depends on the impact any individual can have on the outcome.  We
use the \emph{global sensitivity} of $\phi$ to quantify this impact, which we define for $i = 1,2$ as:
$$
GS_i(\phi) = \max_{\bd,\bd' \text{neighboring in } D^n} \left\{ || \phi(\bd) - \phi(\bd') ||_i\right\}.
$$

\begin{lemma}[\citet{Dwork2006Eurocrypt,DMNS06}]
Let $\phi: D^n \to \R^d$ have global sensitivity $GS_i(\phi)$ for $i = 1,2$.  Then the mechanism $M_\cD: D^n \to \R^d$ where
\mathmath
M_\cD(\bd) = \phi(\bd) + (Z_1,\cdots, Z_d)^T \qquad \{ Z_i\} \stackrel{i.i.d.}{\sim} \cD
\mathmath
is $\epsilon$-differentially private if $\cD = Laplace\left( \frac{GS_1(\phi) }{\epsilon} \right)$ or $(\epsilon,\delta)$-differentially private if $\cD = N(0,\sigma^2)$ with $\sigma = \frac{GS_2(\phi)\sqrt{2\ln(2/\delta)}}{\epsilon}$. 
\label{lem:Lap}
\end{lemma}

There are many useful properties of differentially private mechanisms.  The one we will use in this paper is referred to as \emph{post-processing}, which ensures privacy no matter what we do with the outcome of $M$.
%which means that any adversary cannot take a differentially private outcome and any auxiliary information not based on the original database to then make the outcome any less  private.  
\begin{lemma}\ifnum\submission=1\citep{DMNS06}\else [Post-Processing \citep{DMNS06}]\fi
Let $M:D^n \to O$ be $(\epsilon,\delta)$-differentially private and $\psi: O \to O'$ be some arbitrary mapping from $O$ to $O'$.  Then $\psi \circ M : D^n \to O'$ remains $(\epsilon,\delta)$-differentially private.
\label{lem:post}
\end{lemma}
The tests that we present will be differentially private, assuming $n$ is known and public, because we will add Laplace or Gaussian noise as in \Cref{lem:Lap} to the vector of counts in goodness of fit testing %or to each cell in the contingency table in independence testing and only use these values (as post-processing) in our tests.  

\section{Hypothesis Testing Preliminaries}
Given sampled data from a population, we wish to test whether the data came from a specific model, which is given as a \emph{null hypothesis} $H_0$.  We will denote our test as an algorithm $\cA$ that takes a dataset $\bX$, significance level $1-\alpha$ and null hypothesis $H_0$ and returns a decision of whether to reject $H_0$ or not.  We would like to design our test so that we achieve Type I error at most $\alpha$, that is
\mathmath
\Prob{}{\cA(\bX; \alpha, H_0) = \text{Reject}| H_0} \leq \alpha
\mathmath
while also achieving a small Type II error $\beta = \Prob{}{\cA(\bX; \alpha, H_0) = \text{Reject} | H_1}$ when the model is actually some \emph{alternate} $H_1 \neq H_0$.  Note that the probability is taken over the randomness from the data generation and the possible randomness from the algorithm $\cA$ itself. It is common to refer to $1-\alpha$ as the \emph{significance} of test $\cA$ and $1-\beta$  as the \emph{power} of $\cA$.  We think of bounding Type I error as a \emph{hard constraint} in our tests and then hope to minimize Type II error.  One reason for this is that we want to be able to avoid false discoveries, in which analysts draw incorrect conclusions from a study.  In contrast, there may be some scenarios where controlling for Type II error is more important.  For instance, we may use hypothesis tests for feature selection where it does not matter if we include unimportant features as long as we do not miss any of the important ones.  
 
 \section{Goodness of Fit Test}\label{sec:gof_prelim}
We consider $\mathbf{X}=(X_1, \cdots, X_d)^T \sim
\mult(n,\mathbf{p})$ where $\mathbf{p} = (p_1,\cdots,
p_d)$ and $\sum_{i=1}^d p_i= 1$.  \ifnum\submission=0 Note that the multinomial distribution is the generalization of a binomial distribution where there are $d$ outcomes.  \fi  For a \emph{goodness of fit test}, we want to test the null hypothesis
$H_0: \mathbf{p}= \mathbf{p}^0$.  A common way to test this is based on the \emph{chi-squared} statistic $Q^2$ where
\ifnum\submission=1
$Q^2 = \sum_{i=1}^d \frac{\left(X_i - np_i^0\right)^2}{np_i^0}$.
\else
\begin{equation}
Q^2 = \sum_{i=1}^d \frac{\left(X_i - np_i^0\right)^2}{np_i^0}
\label{eq:Q}
\end{equation}
\fi
We present the classical chi-squared \emph{goodness of fit test} 
in \Cref{alg:goodness_fit}, which compares the chi-squared statistic
$Q^2$ to a threshold $\chi^2_{d-1,1-\alpha}$ that depends on a desired level of significance $1-\alpha$ as well as the dimension of the data. The threshold $\chi^2_{d-1,1-\alpha}$ satisfies the following relationship:
\mathmath
\Prob{}{\chi^2_{d-1} \geq \chi^2_{d-1,1-\alpha} } = \alpha.
\mathmath
where $\chi^2_{d-1}$ is a chi-squared random variable with $d-1$ degrees of freedom\ifnum\submission=0, which is the distribution of the random variable $\mathbf{N}^T\mathbf{N}$ where $\mathbf{N} \sim N(0,I_{d-1})$ \fi.

\begin{algorithm}
\caption{Goodness of Fit Test for Multinomial Data}
\label{alg:goodness_fit}
\begin{algorithmic}
\ifnum\submission=1
\REQUIRE $\GOF(\mathbf{x}, \alpha, H_0: \bp = \bp^0)$
\STATE Compute $Q^2$.
\LINEIF{$ Q^2 > \chi^2_{d-1,1-\alpha}$}{Decision $\gets$ Reject}
\LINEELSE{Decision $\gets $ Fail to Reject}
\RETURN Decision.
\end{algorithmic}
\end{algorithm}
\else
\Procedure{\GOF}{Data $\mathbf{x}$, Significance $1-\alpha$, and $H_0: \bp = \bp^0$}
\State Compute $Q^2$ from \eqref{eq:Q}
\If{$ Q^2 > \chi^2_{d-1,1-\alpha}$}
\State Decision $\gets$ Reject
\Else
\State Decision $\gets $ Fail to Reject
\EndIf
\State {\bf return} Decision.
\EndProcedure
\end{algorithmic}
\end{algorithm}
\fi

The reason why we compare $Q^2$ with the chi-squared distribution is because of the following classical result.  
\begin{theorem}\citep{BFH75}
Assuming $H_0: \bp = \bp^0$ holds, the statistic $Q^2$ converges in distribution to a chi-squared with $d-1$ degrees of freedom, i.e. 
\mathmath
Q^2 \stackrel{D}{\to} \chi^2_{d-1}.
\mathmath
\label{thm:asymp_gof}
\end{theorem}
Note that this does not guarantee that $\Prob{}{Q^2 > \chi^2_{d-1,1-\alpha}}\leq \alpha$ for finite samples, nevertheless the test works well and is widely used in practice. 

\ifnum\submission=0
It will be useful for our purposes to understand why the asymptotic result holds in \Cref{thm:asymp_gof}.  We present the following classical analysis \citep{BFH75} of \Cref{thm:asymp_gof} so that we can understand what adjustments need to be made to find an approximate distribution for a differentially private statistic.  Consider the random vector $\mathbf{U} = (U_1,\cdots, U_d)$ where
\begin{equation}
U_i = \frac{X_i - np_i^0}{\sqrt{np_i^0}} \qquad \forall i \in [d].
\label{eq:U}
\end{equation}  
 We write the covariance matrix for $\mathbf{U}$ as $\Sigma$ where
\begin{equation}
\Sigma  =  I_{d} - \sqrt{\bp^0} \sqrt{\bp^0}^T 
%=\begin{bmatrix}
%    (1-p_1^0) & -\sqrt{p_1^0p_2^0}  & \dots  & -\sqrt{p_1^0p_d^0} \\
%    -\sqrt{p_2^0p_1^0} & (1-p_2^0) & \dots  &-\sqrt{p_2^0p_d^0}  \\
%    \vdots & \vdots & \ddots & \vdots \\
 %   -\sqrt{p_d^0p_1^0} &-\sqrt{p_d^0p_2^0}& \dots  & (1-p_d^0)
%\end{bmatrix} =%\left( \begin{matrix} \sqrt{p_1^0} \\ \vdots \\ \sqrt{p_d^0}	\end{matrix}\right) \left( \begin{matrix} \sqrt{p_1^0} \\ \vdots \\ \sqrt{p_d^0}	\end{matrix}\right)^T.
\label{eq:Sigma}
\end{equation}
and $\sqrt{\bp^0} = (\sqrt{p_1^0},\cdots, \sqrt{p_d^0})^T$.  
%Note the the covariance matrix $\Sigma$ is singular and positive semi-definite.  
By the \emph{central limit theorem} we know that $\mathbf{U}$ converges in distribution to a multivariate normal
\mathmath
\mathbf{U} \stackrel{D}{\to} N(\mathbf{0}, \Sigma) \qquad \text{ as } n \to \infty.
\mathmath

Thus, when we make the assumption that $\mathbf{U}$ is multivariate
normal, then the significance of $\GOF$ given in \Cref{alg:goodness_fit}
is exactly $1-\alpha$.

We show in the following lemma that if a random vector
$\mathbf{U}$ is exactly distributed as multivariate normal then we
get that $Q^2  = \mathbf{U}^T \mathbf{U} \sim \chi^2_{d-1}$.  
%Thus, when we make the assumption that $\mathbf{U}$ is a multivariate
%normal, then the significance of $\GOF$ given in \Cref{alg:goodness_fit}
%is exactly $1-\alpha$.

\begin{lemma}[\citep{BFH75}]
If $\mathbf{U} \sim N(0,\Sigma)$ for $\Sigma$ given in \eqref{eq:Sigma} then $\mathbf{U}^T \mathbf{U} \sim \chi^2_{d-1}. $
\end{lemma}
\begin{proof}
The eigenvalues of $\Sigma$ must be either 0 or 1 because $\Sigma$ is idempotent.  Thus, the number of eigenvalues that are 1 equals trace of $\Sigma$, which is $d-1$.  We then know that there exists a matrix $H \in R^{d \times d-1}$ where $\Sigma = H H^T$ and $H^T H = I_{d-1}$.  Define the random variable $\mathbf{Y} \sim N(\mathbf{0},I_{d-1})$.  Note that $H \mathbf{Y}$ is equal in distribution to $\mathbf{U}$.  We then have 
\begin{align*}
\mathbf{U}^T\mathbf{U} \sim \mathbf{Y}^TH^TH\mathbf{Y} \sim \mathbf{Y}^T \mathbf{Y} \sim \chi_{d-1}^2
\end{align*}
\end{proof}
\fi

\subsection{Differentially Private Chi-Squared Statistic} 
To ensure differential privacy, we add independent noise to each component of $\bX$, which we will either use Laplace or Gaussian noise.  The function $g$ that outputs the counts in the $d$ cells has global sensitivity $GS_1(g) =2$ and $GS_2(g) = \sqrt{2}$ because one individual may move from one cell count (decreasing the count by 1) to another (increasing the cell count by 1).   We then form the \emph{private chi-squared} statistic $Q_\cD^2$ based on the noisy counts, 
\begin{equation}
Q^2_{\cD} = \sum_{i=1}^d \frac{\left(X_i + Z_i - np_i^0\right)^2}{np_i^0}, \qquad \{Z_i\} \stackrel{i.i.d.}\sim \cD
\label{eq:Q_priv}
\end{equation}
where the distributions for the noise that we consider include
\ifnum\submission=1
$\cD = Laplace(2/\epsilon)$ and $\cD = N(0,\sigma^2)$ where $ \sigma \equiv \sigeps = \frac{2\sqrt{ \ln(2/\delta_n)}}{\epsilon}$.
\else
\begin{equation}
\cD = Laplace(2/\epsilon) \qquad \text{ and } \qquad \cD = N(0,\sigma^2)  \ifnum\submission=0 \qquad \text{where} \qquad  \sigeps = \frac{2\sqrt{ \ln(2/\delta_n)}}{\epsilon}. \fi
\label{eq:PD}
\end{equation}
\fi
We will denote the $\epsilon$-differentially private statistic as $Q^2_{Lap}$ and the $(\epsilon,\delta)$-differentially private statistic as $Q^2_{Gauss}$ based on whether we use Laplace or Gaussian noise, respectively.  
Recall that in the original goodness of fit test without privacy in \Cref{alg:goodness_fit} we compare the distribution of $Q^2$ with that of a chi-squared random variable with $d-1$ degrees of freedom.
\ifnum\submission=1
We show in the supplementary materials that $Q_\cD^2$ still has the same asymptotic distribution, with certain conditions on $\delta_n$.  
\else
The following result shows that adding noise to each cell count does not affect this asymptotic distribution.
\begin{lemma}
Fixing $\bp^0>0$, and having privacy parameters $(\epsilon,\delta_n)$ where $\epsilon >0$ and $\delta_n$ satisfies the following condition $\frac{\log(2/\delta_n)}{n \epsilon^2} \to 0$, then the private chi-squared statistic $Q_\cD^2$ given in \eqref{eq:Q_priv} converges in distribution to $\chi^2_{d-1}$ as $n \to \infty$.  
\end{lemma}
\begin{proof}
We first expand \eqref{eq:Q_priv} to get 
\begin{align*}
Q^2_\cD  = \sum_{i=1}^d \left( \frac{X_i - np_i^0}{\sqrt{np_i^0}} \right)^2 + 2\sum_{i=1}^d \left( \frac{Z_i}{\sqrt{n p_i^0}} \right) \left( \frac{X_i-np_i^0}{\sqrt{np_i^0}} \right) + \sum_{i=1}^d \left(\frac{Z_i}{\sqrt{n p_i^0}}\right)^2
\end{align*}
We first focus on $\cD$ being Gaussian.  We define the two random vectors $\bZ^{(n)} = \left(\frac{Z_i}{\sqrt{np_i^0}} \right)_{i=1}^n$ and $\bX^{(n)} =  \left( \frac{X_i - np_i^0}{\sqrt{np_i^0}} \right)_{i=1}^n$.  We have that $\Var(Z_i^{(n)}) = \frac{\sigeps^2}{n p_i^0} = \frac{2\ln(2/\delta_n)}{n p_i^0 \epsilon^2 }$ which goes to zero by hypothesis.  Additionally $\ex{\bZ^{(n)}} = 0$, so we know that $\bZ^{(n)} \stackrel{P}{\to} \mathbf{0}$.  We also know that $\bX^{(n)} \stackrel{D}{\to} N(0,\Sigma)$, so that $\bZ^{(n)} \cdot \bX^{(n)} \stackrel{D}{\to} 0$ by Slutsky's Theorem\footnote{Slutsky's Theorem states that if $X_n \stackrel{D}{\to} X$ and $Z_n \stackrel{P}{\to} c$ then $X_n \cdot Z_n \stackrel{D}{\to} c X $ and $X_n + Z_n \stackrel{D}{\to} X+ c$.} and thus $\bZ^{(n)} \cdot \bX^{(n)} \stackrel{P}{\to} 0$ (because 0 is constant).  Another application of Slutsky's Theorem tells us that $Q_\cD^2 \stackrel{D}{\to}\chi_{d-1}^2$, since $Q_\cD^2 - Q^2 \stackrel{P}{\to} 0$ and $Q^2 \stackrel{D}{\to} \chi^2_{d-1}$.  The proof for $\cD$ being Laplacian follows the same analysis.
\end{proof}
\fi

It then seems natural to use $\GOF$ on the private chi-squared statistic as if we had the actual
chi-squared statistic that did not introduce noise to each count since both private and nonprivate statistics have the same asymptotic distribution.  We
will show in our results in \Cref{sec:results} that if
we were to simply compare the private statistic to the
critical value $\chi^2_{d-1,1-\alpha}$, we will typically not get a
good significance level even for relatively large $n$\ifnum\submission=0  \hspace{1mm} which we need in order for it to be practical
tool for data analysts\fi.  In the following lemma we show that for every realization of data, the statistic $Q_\cD^2$ is expected to be larger than the actual chi-squared statistic $Q^2$. \ifnum\submission=1 See the supplementary materials for the proof.  \fi 

\begin{lemma}
For each realization $\mathbf{X} = \bbx$,  we have $\Ex{\cD}{Q_\cD^2|\bbx} \geq Q^2$, where $\cD$ has mean zero.  
\end{lemma}

\ifnum\submission=0
\begin{proof}
Consider the convex function $f(y) = y^2$.  Applying Jensen's inequality, we have $f(y) \leq \Ex{Z_i}{f(y+Z_i)}$ for all $i = 1,\cdots, d$ where $Z_i$ is sampled i.i.d. from  $\cD$ which has mean zero.  We then have for $\mathbf{X} = \bbx$
\begin{align*}
Q^2 & =  \sum_{i=1}^d\frac{f\left(x_i-np_i^0\right)}{np_i^0} = \sum_{i=1}^d\frac{f\left( \Ex{}{x_i-np_i^0 + Z_i } \right)}{np_i^0} \\
& \leq \Ex{\{Z_i \}\stackrel{i.i.d.}\sim \cD }{\sum_{i=1}^d \frac{f\left(x_i - np_i^0 + Z_i\right)}{np_i^0} } = \Ex{\{Z_i \} \stackrel{i.i.d.}\sim \cD}{Q_\cD^2|\bbx}
\end{align*}
\end{proof}
\fi
This result suggests that the significance threshold for the private version of the chi-squared statistic $Q_\cD^2$ should be higher than the standard one. Otherwise, we would reject $H_0$ too easily using the classical test, which we show in our experimental results.  This motivates the need to develop new tests that account for the distribution of the noise. 

\subsection{Monte Carlo Test: $\MC_\cD$}\label{sec:MC}

Given some null hypothesis $ \bp^0$ and statistic $Q^2_\cD$, we want to determine a threshold $\tau^\alpha$ such that $Q^2_\cD> \tau^\alpha$ at most an $\alpha$ fraction of the time when the null hypothesis is true.  As a first approach, we determine threshold $\tau^\alpha$ using a Monte Carlo (MC) approach by sampling from the distribution of $Q^2_{\cD}$, where $\mathbf{X} \sim \mult(n,\bp^0)$ and $\{Z_i\} \stackrel{i.i.d.}{\sim} \cD$ for both Laplace and Gaussian noise.  
\ifnum\submission=1
We give our MC based test $\MC$ in \Cref{alg:mc}.  We show in the supplementary materials that $\MC$ achieves at least our target significance $1-\alpha$ when we choose $k > 1/\alpha$ many samples.
\else

Let $M_1,\cdots, M_k$ be $k$ continuous random variables that are i.i.d. from the distribution of $Q^2_{\cD}$ assuming $H_0$.   Further let $M$ be a fresh sample from the distribution of $Q^2_\cD$ assuming $H_0$.  \ifnum\submission=0 We will write the density and distribution of $Q^2_\cD$ as $f(\cdot)$ and $F(\cdot)$, respectively.  \fi Our test will reject $M$ if it falls above some threshold, i.e. critical value, which we will take to be the $t$-th order statistic of $\{ M_i\}$, also written as $M_{(t)}$, so that with probability at most $\alpha$, $M$ is above this threshold.  This will guarantee significance at least $1-\alpha$.  We then find the smallest $t \in [k]$ such that $\alpha  \geq \Prob{}{M > M_{(t)} }$, or
 \begin{align*}
 \alpha & \geq \int_{-\infty}^\infty f(m) \sum_{j = t}^k { k \choose j} F(m)^j (1-F(m) )^{k-j} dm = \int_0^1 \sum_{j = t}^k {k \choose j} p ^j (1-p)^{k-j} dp \\
 & = \sum_{j = t}^k \frac{1}{k+1} \implies  t \geq (k+1)(1-\alpha).
\end{align*}

We then set our threshold based on the $\lceil (k+1)(1-\alpha)\rceil$ ordered statistic of our $k$ samples.  By construction, this will ensure that we achieve the significance level we want.  Our test then is to sample $k$ points from the distribution of $Q^2_\cD$ and then take the $\lceil (k+1)(1-\alpha)\rceil$- percentile as our cutoff, i.e. if our statistic falls above this value, then we reject $H_0$.  Note that we require $k\geq 1/\alpha$, otherwise there would not be a $\lceil (k+1)(1-\alpha)\rceil$ ordered statistic in $k$ samples.  We give the resulting test in \Cref{alg:mc}.
\fi
\begin{algorithm}
\caption{MC Goodness of Fit}
\label{alg:mc}
\begin{algorithmic}
\ifnum\submission=1
\REQUIRE $\MC_\cD(\bbx, (\epsilon,\delta),\alpha,H_0: \bp = \bp^0 ) $
\STATE Compute $q =  Q^2_\cD$  \eqref{eq:Q_priv}.
\STATE Select $k > 1/\alpha$.
\STATE Sample $q_1,\cdots, q_k$ i.i.d. from the distribution of $Q^2_\cD$.  
\STATE Sort the samples $q_{(1)}\leq \cdots \leq q_{(k)}$.
\STATE Compute threshold $q_{(t)}$ where $t = \lceil (k+1)(1-\alpha)\rceil$.
\LINEIF{$q > q_{(t)}$}{Decision $\gets$ Reject}
\LINEELSE{Decision $\gets $ Fail to Reject}
\RETURN Decision
\else
\Procedure{$\MC_\cD$}{Data $\bbx$; Privacy $(\epsilon,\delta)$, Significance $1-\alpha$, $H_0: \bp = \bp^0$}
\State Compute $q =  Q^2_\cD$  \eqref{eq:Q_priv}.
\State Select $k > 1/\alpha$.
\State Sample $k$ points $q_1,\cdots, q_k$ i.i.d. from the distribution of $Q^2_\cD$ and sort them $q_{(1)} \leq \cdots \leq q_{(k)}$.  
\State Compute threshold $q_{(t)}$ where $t = \lceil (k+1)(1-\alpha)\rceil$.
\If{$q > q_{(t)}$}
\State Decision $\gets$ Reject
\Else
\State Decision $\gets $ Fail to Reject
\EndIf
\State {\bf return} Decision.
\EndProcedure
\fi
\end{algorithmic}
\end{algorithm}

\begin{theorem}
%The test $\MC_\cD$ is $\epsilon$-differentially private and $(\epsilon,\delta)$-differentially private for $\cD$ being $Laplace$ and $Gauss$, respectively in \ifnum\submission=1 \eqref{eq:Q_priv} \else \eqref{eq:PD}\fi, (assuming the size $n$ of the data is known and public).  
The test $\MC_\cD(\cdot,(\epsilon,\delta),\alpha,\bp^0)$ has significance at least $1-\alpha$, \ifnum\submission=1 i.e. \else also written as \fi $\Prob{}{\MC_\cD(\bX,(\epsilon,\delta),\alpha,\bp^0) = \text{ Reject } | H_0} \leq \alpha$.
\label{thm:mcgof}
\end{theorem}

In \Cref{sec:results}, we present the empirical power results for $\MC_\cD$ (along with all our other tests) when we fix an alternative hypothesis.

\subsection{Asymptotic Approach: $\PG$}
In this section we attempt to determine a more analytical approximation to the distribution of $Q^2_{Gauss}$.  We focus on Gaussian noise because it is more compatible with the asymptotic analysis of $\GOF$, as opposed to Laplace noise.   
\ifnum\submission=0 Recall the random vector $\mathbf{U}$ given in \eqref{eq:U}.  
\else Consider the random vector $\mathbf{U} = (U_1,\cdots, U_d)$ where $U_i = \frac{X_i - np_i^0}{\sqrt{np_i^0}}$ for any  $i \in [d].$  %Note that $\mathbf{U}$ has covariance matrix $
\fi
We then introduce the
Gaussian noise random vector as $\bV = (Z_1/\sigeps,
\cdots, Z_d/\sigeps )^T \sim N(\mathbf{0}, I_d)$. 
 Let $\mathbf{W} \in \R^{2d}$ be the concatenated vector defined as 
 \ifnum\submission=1 $\mathbf{W} = {\mathbf{U} \choose \mathbf{V}}$. 
\else
\begin{equation}
\mathbf{W} = {\mathbf{U} \choose \mathbf{V}}.  
\label{eq:W}
\end{equation}
\fi

Note that $\mathbf{W}\stackrel{D}{\to} N(\mathbf{0},\Sigma')$ where the covariance matrix is the $2d$ by $2d$ block matrix
%%%
\iffalse
We show that the random vector $\mathbf{W} $ converges in distribution to a multivariate normal 
\begin{lemma}
Let $\mathbf{U} \stackrel{D}\to N(\mathbf{0},\Sigma_1)$, $\mathbf{V} \sim N(\mathbf{0},I_d)$, and $\mathbf{U},\mathbf{V}$ are independent.  Then $\mathbf{W}$ as defined in \eqref{eq:W} converges in distribution to $N(0,\Sigma')$ where
\begin{equation}
\Sigma' = \begin{bmatrix}
\Sigma & 0\\
0 & I_d
\end{bmatrix}
\label{eq:sigma_prime}
\end{equation}
and $\Sigma$ is given in \eqref{eq:Sigma}.  Since $\Sigma$ is idempotent, so is $\Sigma'$.
\label{lem:converge_dist}
\end{lemma}
\begin{proof}
We use the Cramer Wold Device, which states that if $\forall \mathbf{t} \in \R^{2d}$ we have $\bbt^T \mathbf{W}\stackrel{D}\to N(\mathbf{0}, \bbt^T\Sigma'\bbt)$ then $\mathbf{W} \stackrel{D}{\to} N(0,\Sigma')$.  Let $\bbt = \begin{bmatrix} \bbt_1 \\ \bbt_2 \end{bmatrix}$.  We then have 
\begin{align*}
\bbt^T \mathbf{W} = \bbt_1^T \mathbf{U} + \bbt_2^T\mathbf{V}.
\end{align*}
We know that if two random variables $X^{(n)}\stackrel{D}\to X$ and $Y^{(n)}\stackrel{D}\to Y$ are independent, then $X^{(n)}+Y^{(n)}\stackrel{D}{\to} X+Y$.  Thus, we have 
$$
\bbt^T \mathbf{W}^{(n)} \stackrel{D}\to N\left(\mathbf{0}, \bbt_1^T \Sigma \bbt + \bbt_2^T \bbt_2 \right) = N(\mathbf{0}, \bbt^T \Sigma' \bbt).
$$
\end{proof}
\fi
%%%

\begin{equation}
\Sigma' = \begin{bmatrix}
\Sigma & 0\\
0 & I_d
\end{bmatrix}
\ifnum\submission=1,\quad \text{ where } \quad \Sigma = I_{d} - \sqrt{\bp^0} \sqrt{\bp^0}^T  \fi
\label{eq:sigma_prime}
\end{equation}
\ifnum\submission=0 where $\Sigma$ is given in \eqref{eq:Sigma}.\fi  Since $\Sigma$ is idempotent, so is $\Sigma'$.  
We next define the $2d$ x $2d$ positive semi-definite matrix $A$ (composed of four $d$ by $d$ block matrices) as 
\begin{equation}
A =  \begin{bmatrix} 
I_d & \Lambda\\
 \Lambda &  \Lambda^2
\end{bmatrix}
\quad \text{ where } \quad  \Lambda= \diag\left(\frac{\sigeps}{\sqrt{n\bp^0}} \right)
%\begin{bmatrix}
% \frac{\sigma}{\sqrt{n p_1^0}} & 0 & \cdots & 0 \\
%0 & \frac{\sigma}{\sqrt{n p_2^0}} & \cdots & 0 \\
%\vdots & \vdots & \ddots & \vdots \\
%0 & 0 & \cdots &  \frac{\sigma}{\sqrt{n p_d^0}} 
%\end{bmatrix}
\label{eq:A}
\end{equation}

We can then rewrite our private chi-squared statistic as a quadratic form 
\ifnum\submission=1 $Q_{Gauss}^2 = \mathbf{W}^T A \mathbf{W}.$\else of the random vectors $\mathbf{W}$.
\begin{equation}
Q_{Gauss}^2 = \mathbf{W}^T A \mathbf{W}.
\label{eq:chisq_quad}
\end{equation}  
\fi

\begin{remark}
If we have $\sigeps/\sqrt{n\bp^0} \to$ constant then the asymptotic distribution of $Q_{Gauss}^2$ would be a quadratic form of multivariate normals.  
\end{remark}

Similar to the classical goodness of fit test we consider the limiting case that the random vector $\mathbf{U}$ is actually a multivariate normal, which will result in $\mathbf{W}$ being multivariate normal as well.  We next want to be able to calculate the distribution of the quadratic form of normals $\mathbf{W}^TA\mathbf{W}$.  Note that we will write $\{\chi^{2,i}_1 \}_{i=1}^r$ as a set of $r$ independent chi-squared random variables with one degree of freedom, so that $\sum_{i = 1}^r \chi^{2,i}_1 = \chi^2_r$.
\begin{theorem}\label{lem:quad_dist}
 Let $\mathbf{W} \sim N(\mathbf{0}, \Sigma')$ where $\Sigma'$ is idempotent and has rank $r\leq 2d$.  Then the distribution of $\mathbf{W}^T A \mathbf{W}$ where $A$ is positive semi-definite is 
\mathmath
\sum_{i=1}^r \lambda_i \chi^{2,i}_{1}
\mathmath
where $\{\lambda_i\}_{i=1}^r$ are the eigenvalues of $B^TAB$ where $B \in \R^{2d\times r}$ such that $BB^T = \Sigma'$ and $B^T B = I_r$.  
\end{theorem}
\ifnum\submission=0
\begin{proof}
Let $\mathbf{N}^{(1)} \sim N(\mathbf{0},I_r)$.  Because $\Sigma'$ is idempotent, we know that there exists a matrix $B \in \R^{2d \times r}$ as in the statement of the lemma. Then $B\mathbf{N}^{(1)}$ has the same distribution as $\mathbf{W}$.  Also note that because $B^TAB$ is symmetric, then it is diagonalizable and hence there exists an orthogonal matrix $H \in \R^{r\times r}$ such that 
$$
H^T (B^T A B) H = \diag(\lambda_1,\cdots, \lambda_r) \qquad \text{ where } \qquad H^T H = H H^T= I_r
$$
Let $\mathbf{N}^{(1)} = H \mathbf{N}^{(2)}$ where $\mathbf{N}^{(2)} \sim N(\mathbf{0},I_r)$.  We then have 
$$
\mathbf{W}^T A \mathbf{W} \sim \left( B \mathbf{N}^{(1)}\right)^T A \left( B \mathbf{N}^{(1)} \right) \sim  \left( B H \mathbf{N}^{(2)}\right)^T A \left( B H \mathbf{N}^{(2)} \right) \sim \left(\mathbf{N}^{(2)}\right)^T\diag(\lambda_1,\cdots\lambda_r) \mathbf{N}^{(2)}
$$
%for all $t \geq 0$
%$$
%\Prob{}{\mathbf{W}^T A \mathbf{W} \leq t} = \Prob{}{(BH\mathbf{N}^{(2)})^T A (BH\mathbf{N}^{(2)}) \leq t} = \Prob{}{(\mathbf{N}^{(2)})^T Diag(\lambda_1,\cdots, \lambda_r) \mathbf{N}^{(2)} \leq t}.
%$$
Now we know that $(\mathbf{N}^{(2)})^T \diag(\lambda_1,\cdots, \lambda_r) \mathbf{N}^{(2)}\sim \sum_{j=1}^r \lambda_j \chi_{1}^{2,j}$, which gives us our result.
\end{proof}
\fi

Note that in the non-private case, the coefficients $\{\lambda_i\}$ in  \Cref{lem:quad_dist} become the eigenvalues for the rank $d-1$ idempotent matrix $\Sigma$, thus resulting in a $\chi^2_{d-1}$ distribution.  We use the result of \Cref{lem:quad_dist} in order to find a threshold
that will achieve the desired significance level $1-\alpha$, as in the
classical chi-squared goodness of fit test.  We then set the threshold
$\tau^\alpha$ to satisfy the following: 
\begin{equation}
\Prob{}{\sum_{i=1}^r \lambda_i \chi^{2,i}_1 \geq \tau^\alpha} = \alpha
\label{eq:thresh_test}
\end{equation}
for $\{\lambda_i\}$ found in \Cref{lem:quad_dist}.  Note, the
threshold $\tau^\alpha$ is a function of
$n,\epsilon,\delta,\alpha$ and $\bp^0$, but not the data.  

We present our modified goodness of fit test when we are dealing with differentially private counts in \Cref{alg:priv_good}.

\ifnum\submission=1
\vspace{-1mm}
\fi
\begin{algorithm}
\caption{ Private Chi-Squared Goodness of Fit Test}
\label{alg:priv_good}
\begin{algorithmic}
\ifnum\submission=1
\REQUIRE $\PG(\mathbf{x},(\epsilon,\delta), \alpha, H_0: \bp = \bp^0)$
\STATE Set $\sigma = 2\sqrt{\log(2/\delta)}/\epsilon$.
\STATE Compute $Q^2_{Gauss}$ from \eqref{eq:Q_priv} and $\tau^\alpha$ that satisfies \eqref{eq:thresh_test}.
\LINEIF{$ Q^2_{Gauss} > \tau^\alpha$}{Decision $\gets$ Reject}
\LINEELSE{Decision $\gets $ Fail to Reject}
\RETURN Decision
\else
\Procedure{\PG}{Data $\mathbf{x}$; Privacy $(\epsilon,\delta)$, Significance $1-\alpha$, $H_0: \bp = \bp^0$}
\State Set $\sigma = \frac{2\sqrt{\log(2/\delta)}}{\epsilon}$.
\State Compute $Q^2_{Gauss}$ from \eqref{eq:Q_priv}
\State Compute $\tau^\alpha$ that satisfies \eqref{eq:thresh_test}.
\If{$ Q^2_{Gauss} > \tau^\alpha$}
\State Decision $\gets$ Reject
\Else
\State Decision $\gets $ Fail to Reject
\EndIf
\State {\bf return} Decision.
\EndProcedure
\fi
\end{algorithmic}
\end{algorithm}

\ifnum\submission=1
\vspace{-2mm}
\fi

%\ifnum\submission=1
%In the supplementary materials, we also give the asymptotic distribution for the statistic $Q^2_{Gauss}$ when the data comes from an alternate probability vector $\bp^1$ that is close to the null hypothesis vector $\bp^0$.  This result is to be compared with the classical result of \citet{Mitra58}, which shows that under a specific type of alternate $\bp^1 \neq \bp^0$, the statistic $Q^2$ converges in distribution to a noncentral chi-squared with $d-1$ degrees of freedom.
%\else
\ifnum\submission=0
\subsection{Power Analysis of $\PG$}
To determine the power of our new goodness of fit test $\PG$, we need to specify an alternate hypothesis $H_1:\bp = \bp^1$ for $\bp^1 \neq \bp^0$.  Similar to past works \citep{Mitra58,MC66, Guenther77}, we define parameter $\tilde\Delta >0$ where
\begin{equation}
\bp^1_n = \bp^0 +\frac{\tilde\Delta}{\sqrt{n}} (1,-1, \cdots, -1, 1)^T
\label{eq:p1tilde}
\end{equation}
for even $d$.  Note that $Q^2_{Gauss}$ uses the probability vector given in $H_0$ but data is generated by $\mult(n,\bp^1_n)$.  
\ifnum\submission=0
In fact, the nonprivate statistic $Q^2$ when the data is drawn from $H_1$ no longer converges to a chi-squared distribution.  Instead, $Q^2$ converges in distribution to a noncentral chi-squared when $H_1$ holds.\footnote{Note that a noncentral chi-squared with noncentral parameter $\theta$ and $\nu$ degrees of freedom is the distribution of $\bX^T\bX$ where each $\bX \sim N(\mu,I_\nu)$ and $\theta = \mu^T \mu$.} 
\begin{lemma}[\citet{BFH75}]
Under the alternate hypothesis $H_1: \bp= \bp^1_n$ given in \eqref{eq:p1tilde}, the chi-squared statistic $Q^2$ converges in distribution to a noncentral $\chi^2_{d-1}(\nu)$ where $\nu = \frac{\tilde\Delta^2}{\prod_{i=1}^d p_i^0}$, i.e. given $H_1: \bp = \bp^1_n$ we have
\mathmath
Q^2 \stackrel{D}{\to} \chi^2_{d-1} (\nu) \qquad \text{ as } n \to \infty.
\mathmath
\label{lem:noncentral}
\end{lemma}
\fi
Another classical result tells us that the vector $\mathbf{U}$ from \eqref{eq:U} converges in distribution to a multivariate normal under the alternate hypothesis.
\begin{lemma}[\citep{Mitra58, Mitra55}]
Assume $\mathbf{X} \sim \mult(n,\bp^1_n)$ where $\bp^1_n$ satisfies \eqref{eq:p1tilde}.  Then $\mathbf{U} \stackrel{D}{\to} N(\mu,\Sigma)$ where $\Sigma$ is given in \eqref{eq:Sigma} and
\ifnum\submission=1
$\mu = \left( \frac{\tilde\Delta}{\sqrt{p^0_1}}, \frac{-\tilde\Delta}{\sqrt{p^0_2}}, \cdots, \frac{-\tilde\Delta}{\sqrt{p^0_{d-1}}}, \frac{\tilde\Delta}{\sqrt{p^0_d}} \right)$.
\else
\begin{equation}
\mu = \left( \frac{\tilde\Delta}{\sqrt{p^0_1}}, \frac{-\tilde\Delta}{\sqrt{p^0_2}}, \cdots, \frac{-\tilde\Delta}{\sqrt{p^0_{d-1}}}, \frac{\tilde\Delta}{\sqrt{p^0_d}} \right)
\label{eq:mu_alt}
\end{equation}
\fi
\label{lem:power_vector}
\end{lemma}

\begin{corollary}
Under the alternate hypothesis $H_1: \bp = \bp^1_n$, then the random vector $\mathbf{W} \stackrel{D}{\to} N(\mu',\Sigma')$   for $\mathbf{W}$ given in \eqref{eq:W} where
$\mu' = (\mu,\mathbf{0})^T $ and $\mu,\Sigma'$ given in \eqref{eq:mu_alt} and \eqref{eq:sigma_prime}, respectively.
\label{eq:Walt}
\end{corollary}
We then write our private statistic as $Q_{Gauss}^2 = \mathbf{W}^T A \mathbf{W}$.  
Similar to the previous section we will write $\{\chi^{2,i}_1(\nu_j) \}_{j=1}^r$ as a set of $r$ independent noncentral chi-squareds with noncentral parameter $\nu_j$ and one degree of freedom.  \ifnum\submission=1 We prove the following result in the supplementary materials, which follows a similar analysis as the work of \cite{QuadForm}. \fi

\begin{theorem}
Let $\mathbf{W} \sim N(\mu',\Sigma')$ where $\mu'$ and $\Sigma'$ are given in \Cref{eq:Walt}.  We will write $\Sigma' = BB^T$ where $B \in \R^{2d \times (2d-1)}$ has rank $2d-1$ and $B^T B = I_{2d-1}$.  We define $\bbb^T = (\mu')^TABH$ where $H$ is an orthogonal matrix such that $H^T B^TABH = \diag(\lambda_1,\cdots, \lambda_{2d-1})$ and $A$ is given in \eqref{eq:A}.  Then we have
\begin{equation}
\mathbf{W}^TA \mathbf{W}  \sim \sum_{j=1}^{r} \lambda_j \chi^{2,j}_1(\nu_j) +N\left(\kappa,\sum_{j=r+1}^d4b_j^2\right),
\label{eq:power_gof}
\end{equation}
where $(\lambda_j)_{j=1}^{2d-1}$ are the eigen-values of $B^TAB$ such that $\lambda_1 \geq \lambda_2 \geq \lambda_{r} > 0 = \lambda_{r+1} = \cdots =  \lambda_{2d-1} $ and 
$$
\nu_j = \left(\frac{b_j}{\lambda_j}\right)^2 \quad \text{ for } j \in [r] \quad\&\quad \kappa =\frac{\tilde\Delta^2}{\prod_{i=1}^d p_i^0}- \sum_{j=1}^{r} \frac{b_j^2}{\lambda_j} .
$$  
\label{lem:priv_power}
\end{theorem}

\ifnum\submission=0
\begin{proof}
We follow a similar analysis as \cite{QuadForm} for finding the distribution of a quadratic form of normals.  Consider the random variable $\mathbf{N}^{(2)} = B H\mathbf{N}^{(1)} + \mu'$ where $\mathbf{N}^{(1)} \sim N(\mathbf{0},I_{2d-1})$.  Note that $\mathbf{N}^{(2)}$ has the same distribution as $\mathbf{W}$.  We then have for $t \geq 0$
\begin{align*}
\Prob{}{\mathbf{W}^T A \mathbf{W} \geq t} &= \Prob{}{(\mathbf{N}^{(1)})^TH^T B^T A B H \mathbf{N}^{(1)} +2(\mu')^T ABH\mathbf{N}^{(1)} + (\mu')^T A \mu'\geq t}  \\
& = \Prob{}{(\mathbf{N}^{(1)})^T\diag(\lambda_1,\cdots, \lambda_{d+1}) \mathbf{N}^{(1)} + 2 \bbb^T \mathbf{N}^{(1)} + (\mu')^T A \mu'\geq t} \\
& = \Prob{}{\sum_{j=1}^{r} \lambda_j \cdot \left(N^{(1)}_j + b_j/\lambda_j \right)^2  + \sum_{j = r+1}^d2 b_j N_j^{(1)}+\kappa \geq t} \\
& = \Prob{}{\sum_{j=1}^{d} \lambda_j \cdot \chi^{2,j}_1\left( \left( \frac{b_j}{\lambda_j}\right)^2\right) + N\left(0,\sum_{j=r+1}^d4b_j^2\right)+ \kappa \geq t}
\end{align*}
\end{proof}
\fi
\begin{remark}
Again, if we have $\sigeps/\sqrt{n \bp^0} \to$ constant then the asymptotic distribution of $Q_{Gauss}^2$ converges in distribution to the random variable of the form given in \eqref{eq:power_gof} when $H_1$ from \eqref{eq:p1tilde} is true.
\end{remark}

Obtaining the asymptotic distribution for $\hat Q_{Gauss}^2$ when the alternate hypothesis holds may allow for future results on \emph{effective sample size}, i.e. how large a sample size needs to be in order for $\PG$ to have Type II error at most $\beta$ against $H_1: \bp = \bp_n^1$.  We see this as an important direction for future work.
\fi

\section{Independence Testing}

We now consider the problem of testing whether two random variables
$\YY{1} \sim \mult(1,\pipi{1})$ and $\YY{2} \sim
\mult(1,\pipi{2})$ are independent of each other.  
\ifnum\submission=0
Note that $\sum_{i=1}^r \pipi{1}_i = 1$ and $\sum_{j=1}^c \pi^{(2)}_j = 1$, so we can write $\pi^{(1)}_r = 1- \sum_{i<r} \pi_i^{(1)}$ and $\pi_c^{(2)} = 1 - \sum_{j<c} \pi_j^{(2)}$.  
\fi
%Thus $\pi^{(1)} \in \R_{> 0}^{m-1}$ and $\pi^{(2)} \in \R_{>0}^{n-1}$.  
We then form the null hypothesis $H_0: \YY{1}\bot\YY{2}$, i.e. they are independent.  One approach to testing $H_0$ is to 
sample $n$ joint outcomes of $\YY{1}$ and $\YY{2}$ and count the
number of observed outcomes, $X_{i,j}$ which is the number of times
$Y_i^{(1)}= 1$ and $Y_j^{(2)} = 1$ in the $n$ trials, so that we can summarize all joint outcomes as a contingency table $\bX=(X_{i,j}) \sim \mult (n,\bp)$, where $p_{i,j}$ is the probability that
$Y^{(1)}_i = 1$ and $Y^{(2)}_j = 1$.  
\ifnum\submission=0
In \Cref{table:contingency} we give
a $r\times c$ contingency table giving the number of joint outcomes for the variables $\YY{1}$ and $\YY{2}$ from $n$ independent trials. 
\fi We will write the full contingency table of counts $\bX = (X_{i,j})$ as a vector with the ordering convention that we start from the top row and move from left to right across the contingency table.  

\ifnum\submission=0
\begin{table}[h]
\caption{Contingency Table with Marginals.} %title of the table
\centering % centering table
\begin{tabular}{ c | c | c |c | c | c}
\hline
$Y^{(1)}  \backslash  Y^{(2)}$ & 1 & 2 & $\cdots$ & $c$ & Marginals \\
\hline
1 & $X_{11}$ & $X_{12}$ & $\cdots$ & $X_{1c}$ & $X_{1,\cdot}$ \\
\hline
2 &$X_{21}$ &$X_{22}$ & $\cdots$ & $X_{2c}$ & $X_{2,\cdot}$ \\
\hline
$\vdots$ &$\vdots$ & $\vdots$& $\ddots$& $\vdots$ & $\vdots$ \\
\hline
$r$ &$X_{r,1}$ &$X_{r,2}$ &$\cdots$ &$X_{r,c}$ & $X_{r,\cdot}$ \\
\hline
Marginals & $X_{\cdot, 1}$&$X_{\cdot, 2}$ & $\cdots$ & $X_{\cdot, c}$& $n$
\end{tabular}
\label{table:contingency}
\end{table}
\fi

We want to calculate the chi-squared statistic as in \ifnum\submission=0 \eqref{eq:Q} \else the goodness of fit section \fi
(where now the summation is over all joint outcomes $i$ and $j$), but
now we do not know the true proportion $\bp = (p_{i,j}) $ which depends on $\pipi{1}$ and $\pipi{2}$. 
However,  we can use the \emph{maximum likelihood estimator} (MLE) $\hat\bp$ for the probability vector $\bp$ subject to $H_0$ to form the statistic $\hat Q^2$ \ifnum\submission = 1 $ =  \sum_{i,j} \frac{\left(X_{i,j} - n \hat p_{i,j}\right)^2}{n\hat p_{i,j}}.$ \else where
\begin{equation}
\hat Q^2 = \sum_{i,j} \frac{\left(X_{i,j} - n \hat p_{i,j}\right)^2}{n\hat p_{i,j}}.
\label{eq:hatQ}
\end{equation}

The intuition is that if the test rejects even when the most likely probability vector that satisfies the null hypothesis was chosen, then the test should reject against all others.  
\fi
Note that under the null hypothesis we can write $\bp$ as a function of $\pi^{(1)}$ and $\pi^{(2)}$,  
\begin{equation}
\bp = \mathbf{f}(\pi^{(1)}, \pi^{(2)})  \text{ where }    \ifnum\submission=0 \bbf = (f_{i,j})\quad \text{ and} \quad \fi f_{i,j}(\pi^{(1)},\pi^{(2)}) = \pipi{1}_i \cdot \pipi{2}_j.
\label{eq:f}
\end{equation}
Further, we can write the MLE $\hat\bp$ as described below.

\begin{lemma}[\citep{BFH75}]
Given $\bX$, which is $n$ samples of joint outcomes of $\YY{1} \sim \mult(1,\pi^{(1)})$ and $\YY{2}\sim \mult(1,\pi^{(2)})$, if $\YY{1} \bot \YY{2}$, then the MLE for $\bp = \bbf(\pi^{(1)},\pi^{(2)})$ for $\bbf$ given in \eqref{eq:f} is the following: $\hat \bp  = \bbf(\hat \pi^{(1)},\hat \pi^{(2)})$ where 
\begin{align}
\hat \pi_i^{(1)} = X_{i,\cdot}/n, \text{ } \hat \pi_j^{(2)} = X_{\cdot, j}/n \quad \text{ for } i \in [r],  j \in [ c]
\label{eq:MLE}
\end{align}
where $X_{i, \cdot}= \sum_{j = 1}^r X_{i,j}$ and $X_{\cdot, j} = \sum_{i = 1}^c X_{i,j}$.
\label{lem:MLE}
\end{lemma}

We then state another classical result that gives the asymptotic distribution of $\hat Q^2$ given $H_0$.  
\begin{theorem}\cite{BFH75} \label{cor:ind_orig}
Given the assumptions in \Cref{lem:MLE}, the 
statistic \ifnum\submission=0 $\hat Q^2$ given in \eqref{eq:hatQ} converges in
distribution to a chi-squared distribution, i.e. \fi
\mathmath
\hat Q^2 \stackrel{D}{\to} \chi^2_{\nu}
\mathmath 
for $\nu = (r-1)(c-1)$.
\end{theorem}

\begin{algorithm}
\caption{Pearson Chi-Squared Independence Test}
\label{alg:ind_test}
\begin{algorithmic}
\ifnum\submission=1
\REQUIRE $\IND(\mathbf{x}, \alpha)$
\STATE $\hat\bp \gets $  MLE calculation in \eqref{eq:MLE}
\STATE Compute $\hat Q^2$ \ifnum\submission=0 from \eqref{eq:hatQ}\fi and set $\nu = (r-1)(c-1)$.
\LINEIF{$\hat Q^2 > \chi^2_{\nu, 1- \alpha}$ and all entries of $\bbx$ are at least $5$ }{Decision $\gets$ Reject}
\LINEELSE{Decision $\gets $ Fail to Reject}
\RETURN Decision.
\else
\Procedure{$\IND$}{Data $\mathbf{x}$, and Significance $1-\alpha$}
\State $\hat\bp \gets $  MLE calculation in \eqref{eq:MLE}
\State Compute $\hat Q^2$ from \eqref{eq:hatQ} and set $\nu = (r-1)(c-1)$.
\If{$\hat Q^2 > \chi^2_{\nu, 1- \alpha}$ and all entries of $\bbx$ are at least $5$}
\State Decision $\gets$ Reject
\Else
\State Decision $\gets $ Fail to Reject
\EndIf
\State {\bf return} Decision.
\EndProcedure
\fi
\end{algorithmic}
\end{algorithm}

The chi-squared independence test is then to compare the statistic
$\hat Q^2$, with the value
$\chi^2_{\nu,1-\alpha}$  for a $1-\alpha$ significance test. 
We formally give the Pearson Chi-Squared test in \Cref{alg:ind_test}. An often used ``rule of thumb" \citep{StatBook} with this test is that it can only be used if all the cell counts are at least 5, otherwise the test Fails to Reject $H_0$.  We will follow this rule of thumb in our tests.

Similar to our prior analysis for goodness of fit, we aim to understand the asymptotic distribution from \Cref{cor:ind_orig}.  First, we can define $\hat{\mathbf{U}}$ in terms of the MLE $\hat \bp$ given in \eqref{eq:MLE}:
\begin{equation}
\hat U_{i,j} = (X_{i,j} - n\hat p_{i,j})/\sqrt{n \hat p_{i,j}}.
\label{eq:hatU}
\end{equation}
The following classical result gives the asymptotic distribution of $\mathbf{\hat U}$ under $H_0$, which also proves \Cref{cor:ind_orig}.
\begin{lemma}\citep{BFH75}
With the same hypotheses as \Cref{lem:MLE}, the random vector \ifnum\submission=0 $\mathbf{\hat U}$ given in \eqref{eq:hatU} converges in distribution to a multivariate normal,\fi
\mathmath
\hat{\mathbf{U}} \stackrel{D}{\to} N\left(0,  \Sigma_{ind} \right)
\mathmath
where $\Sigma_{ind} = I_{rc} - \sqrt{\bp} \cdot \sqrt{\bp}^T - \Gamma( \Gamma^T \Gamma)^{-1} \Gamma^T$ with $\bbf$ given in \eqref{eq:f}, and
\mathmath
\Gamma =  \diag(\sqrt{\bp})^{-1}  \cdot \nabla \bbf(\pi^{(1)},\pi^{(2)}) \ifnum\submission=1 .\else ,\fi%, \qquad D = \begin{bmatrix}
												%	\frac{1}{\sqrt{p_{11}}} & 0 & \cdots & 0 \\
												%	0 & \frac{1}{\sqrt{p_{12}}} & \cdots & 0 \\
												%	\vdots & \vdots  & \ddots & \vdots \\
												%	0 & 0 & \cdots & \frac{1}{\sqrt{p_{rc}}}
												%	\end{bmatrix}_{r+c,r+c} , 
\mathmath
\ifnum\submission=0
$$
 \nabla \bbf(\pi^{(1)},\pi^{(2)}) =   \begin{bmatrix}	\frac{\partial f_{1,1}}{\partial \pipi{1}_1}& \cdots & \frac{\partial f_{1,1}}{\partial \pipi{1}_{r-1}} & \frac{\partial f_{1,1}}{\partial \pipi{2}_1}& \cdots & \frac{\partial f_{1,1}}{\partial \pipi{2}_{c-1}} \\
\frac{\partial f_{1,2}}{\partial \pipi{1}_1}& \cdots & \frac{\partial f_{1,2}}{\partial \pipi{1}_{r-1}}& \frac{\partial f_{1,2}}{\partial \pipi{2}_{1}}& \cdots & \frac{\partial f_{1,2}}{\partial \pipi{2}_{c-1}} \\
\vdots & \vdots & \vdots & \vdots & \ddots & \vdots \\
\frac{\partial f_{r,c}}{\partial \pipi{1}_{1}}& \cdots & \frac{\partial f_{r,c}}{\partial \pipi{1}_{r-1}} & \frac{\partial f_{r,c}}{\partial \pipi{2}_1}& \cdots & \frac{\partial f_{r,c}}{\partial \pipi{2}_{c-1}}
													 \end{bmatrix}_{rc,r+c -2}.
$$
\fi
\label{lem:MLE_dist}
\end{lemma}

In order to do a test that is similar to $\IND$ given in \Cref{alg:ind_test}, we need to determine an estimate for $\pipi{1}$ and $\pipi{2}$ where we are only given access to the noisy cell counts.

\subsection{Estimating Parameters with Private Counts}
We now assume that we do not have access to the counts $X_{i,j}$ \ifnum\submission=0 from \Cref{table:contingency} \else in a contingency table \fi but instead we have  $W_{i,j}= X_{i,j} +Z_{i,j}$ where $Z_{i,j} \sim \cD$ for Laplace or Gaussian noise given in \ifnum\submission=1 \eqref{eq:Q_priv} \else\eqref{eq:PD} \fi and we want to perform a test for independence.
We consider the full likelihood of the noisy $r \times c$ contingency table
\ifnum\submission=1
$\Prob{}{\bX+\bZ = \mathbf{w} | H_0, \pi^{(1)},\pi^{(2)}} $
\else
\begin{align*}
& \Prob{}{\bX+\bZ = \mathbf{w} | H_0, \pi^{(1)},\pi^{(2)}} 
= \sum_{\substack{\bbx: \sum_{i,j} x_{i,j} = n \\ x_{i,j} \in \N}} \Prob{}{\bX = \bbx | \pi^{(1)},\pi^{(2)}} \cdot \Prob{}{\bZ = \mathbf{w} - \bbx | X = \bbx} \\
& \qquad =   \sum_{\substack{\bbx: \sum_{i,j} x_{i,j} = n \\ x_{i,j} \in \N}}\underbrace{\Prob{}{\bX = \bbx | \pi^{(1)},\pi^{(2)}} }_{\text{Multinomial}} \prod_{i,j}\underbrace{\Prob{}{Z_{i,j} = w_{i,j} - X_{i,j} | \bX = \bbx} }_{\text{Noise}}
\end{align*}
\fi
to find the best estimates for $\{\pipi{i}\}$ given the noisy counts.

\begin{algorithm}
\caption{ Two Step MLE Calculation}
\label{alg:2MLE}
\begin{algorithmic}
\ifnum\submission=1
\REQUIRE $\MLE(\bX+\bZ = \mathbf{w})$
\STATE $\tilde\bbx \gets$ Solution to \eqref{eq:GaussMLE}.  
\LINEIF{$\cD = $ Gauss}{set $\gamma = 1$}
\LINEIF{$\cD = \Lap$}{ set $0<\gamma\ll 1$}.
\LINEIF{ Any cell of $\tilde\bbx$ is less than 5}{ $\tpipi{1}, \tpipi{2} \gets \NULL$}
\LINEELSE{$\tpipi{1}, \tpipi{2} \gets $ MLE with $\tilde \bbx$ given in \eqref{eq:MLE}.}
\RETURN $\tpipi{1}$ and $\tpipi{2}$.
\else
\Procedure{$\MLE$}{Noisy Data $\bX+\bZ = \mathbf{w}$}
\State $\tilde\bbx \gets$ Solution to \eqref{eq:GaussMLE}.  If $\cD = $ Gauss, then $\gamma = 1$, if $\cD = \Lap$ set $0<\gamma\ll1$.
\State
\If{ Any cell of $\tilde \bbx$ is less than 5}
\State $\tilde{\pipi{1}}, \tilde{\pipi{2}} \gets \NULL$
\Else
\State $\tpipi{1}, \tpipi{2} \gets $ MLE for $\pi^{(1)}$ and $\pi^{(2)}$ with data $\tilde \bbx$.
\State $\qquad$ Note that the MLE for the probabilities $\pi^{(1)}$ and $\pi^{(2)}$ with data is given in \eqref{eq:MLE}.
\EndIf
\State {\bf return} $\tpipi{1}$ and $\tpipi{2}$.
\EndProcedure
\fi
\end{algorithmic}
\end{algorithm}

\ifnum\submission=0
Maximizing this quantity is computationally very expensive for values of $n>100$ even for $2 \times 2$ tables,\footnote{Note that there is a $\poly(n)$ time algorithm to solve this, but the coefficients in each term of the sum can be very large numbers, with $\poly(n)$ bits, which makes it difficult for numeric solvers.} so we instead
\else We will 
\fi follow a two step procedure similar to the work of \cite{KS16}\ifnum\submission=0, where they ``denoise" a private degree sequence for a synthetic graph and then use the denoised estimator to approximate the parameters of the $\beta$-model of random graphs\fi.  We will first find the most likely contingency table given the noisy data $\mathbf{w}$ and then find the most likely probability vectors under the null hypothesis that could have generated that denoised contingency table (this is not equivalent to maximizing the full likelihood, but it seems to work well as our experiments later show).  For the latter step, we use \Cref{eq:MLE} to get the MLE for $\pipi{1}$ and $\pipi{2}$ given a vector of counts $\bbx$.  For the first step, we need to minimize $|| \mathbf{w} - \bbx||$ subject to $\sum_{i,j} x_{i,j} = n$ and $x_{i,j} \geq 0$  where the norm in the objective is either $\ell_1$ for Laplace noise or $\ell_2$ for Gaussian noise.

Note that for Laplace noise, the above optimization problem does not give a unique solution and it is not clear which contingency table $\bbx$ to use.  One solution to overcome this is to add a \emph{regularizer} to the objective value.  We will follow the work of \cite{KLW15} to overcome this problem by using an \emph{elastic net} regularizer \citep{HZ05}:
\begin{align}
\myargmin_{\bbx}  &\qquad (1-\gamma) \cdot || \mathbf{w} - \bbx||_1 + \gamma\cdot || \mathbf{w} - \bbx ||_2^2 \label{eq:GaussMLE}\\
s.t. & \qquad \sum_{i,j} x_{i,j} = n,  \qquad  x_{i,j} \geq 0 \nonumber.
\end{align}
where if we use Gaussian noise, we set $\gamma = 1$ and if we use Laplace noise then we pick a small $\gamma>0$ and then solve the resulting program. Our two step procedure for finding an approximate MLE for $\pi^{(1)}$ and $\pi^{(2)}$ based on our noisy vector of counts $\mathbf{w}$ is given in \Cref{alg:2MLE}, where we take into account the rule of thumb from $\IND$ and return $\NULL$ if any computed table has counts less than $5$.

We will denote $\tilde\bp$ to be the probability vector of $\bbf$ from \eqref{eq:f} applied to the result of $\MLE(\bX+\bZ)$.    We now write down the private chi-squared statistic when we use the estimate $\tilde\bp$ in place of the actual (unknown) probability vector $\bp$:

\begin{equation}
\tilde Q_\cD^2 = \sum_{i,j} \frac{\left(X_{i,j} + Z_{i,j} - n \tilde{p}_{i,j}\right)^2}{n \tilde{p}_{i,j}} \quad \{Z_{i,j} \} \stackrel{i.i.d.}{\sim} \cD.
\label{eq:hatQ_priv}
\end{equation}

\begin{algorithm}
\caption{ MC Independence Testing}
\label{alg:MC_ind}
\begin{algorithmic}
\ifnum\submission=1
\REQUIRE $\MCIND(\bbx,(\epsilon,\delta),\alpha)$
\STATE $\mathbf{w} \gets \bbx+ \bZ$, where $\{ Z_{i,j}\} \stackrel{i.i.d.}{\sim} \cD$ and $\cD$ given in \eqref{eq:Q_priv}.
\STATE $(\tpipi{1}, \tpipi{2} ) \gets \MLE(\mathbf{w})$ and $\tilde\bp \gets \bbf(\tpipi{1}, \tpipi{2} )$.
\IF{$(\tpipi{1}, \tpipi{2} ) == \NULL$}
\LINERETURN Fail to Reject
\ELSE
\STATE $\tilde q \gets \tilde Q^2_\cD$ using $\mathbf{w}$ and $\tilde\bp$.
\STATE Set $k > 1/\alpha$ and $q \gets \NULL$.
\FOR{ $t \in [k]$}
	\STATE Generate contingency table $\tilde\bbx$ using $\left(\tpipi{1}, \tpipi{2} \right)$.  
	\STATE $\mathbf{\tilde w} \gets \tilde \bbx+ \bZ$, where $\{ Z_{i,j}\} \stackrel{i.i.d.}{\sim} \cD$.
	\STATE $(\ttpipi{1}, \tt\pipi{2}) \gets \MLE(\tilde{\mathbf{w}})$.
	\LINEIF{$(\ttpipi{1}, \ttpipi{2}) ==\NULL$}{\LINERETURN Fail to Reject.}
	\LINEELSE  Compute $\tilde Q^2_\cD$ from \eqref{eq:hatQ_priv}, add it to array $q$.
\ENDFOR
\STATE $\tilde\tau^\alpha\gets $ the $\lceil (k+1)(1-\alpha)\rceil$ ranked statistic in $q$.  
\LINEIF{$\tilde q > \tilde\tau^\alpha  $}{$\quad \return$ Reject $H_0$.}
\LINEELSE{$\quad \return$ Fail to Reject $H_0$.}
\ENDIF
\else
\Procedure{$\MCIND$}{Contingency Table $\bbx$; privacy parameters $(\epsilon,\delta)$, significance $1-\alpha$}
\State $\mathbf{w} \gets \bbx+ \bZ$, where $\{ Z_{i,j}\} \stackrel{i.i.d.}{\sim} \cD$ and $\cD$ given in \eqref{eq:PD}.
\State $(\tpipi{1}, \tpipi{2} ) \gets \MLE(\mathbf{w})$.
\If{$(\tpipi{1}, \tpipi{2} ) ==\NULL$}
	\Return Fail to Reject
\Else
\State $\tilde q \gets \tilde Q^2_\cD$ with data $\mathbf{w}$ and parameters $\tilde\bp = \bbf(\tpipi{1}, \tpipi{2} )$.
\State Set $k > 1/\alpha$ and $q \gets \NULL$.
\For{ $t \in [k]$}
	\State Generate a fresh contingency table $\tilde\bbx$ using parameters $\left(\tpipi{1}, \tpipi{2} \right)$.  
	\State $\mathbf{\tilde w} \gets \tilde \bbx+ \bZ$, where $\{ Z_{i,j}\} \stackrel{i.i.d.}{\sim} \cD$ and $\cD$ given in \eqref{eq:PD}.
	\State $\ttpipi{1}, \ttpipi{2} \gets \MLE(\tilde{\mathbf{w}})$.
	\If{$\ttpipi{1}, \ttpipi{2}  == \NULL$}
		\Return Fail to Reject
	\Else
	\State Concatenate $q$ with $\tilde Q^2_\cD$ given in \eqref{eq:hatQ_priv} with $\tilde{\mathbf{w}}$ and 	
	$\tilde \bp = \bbf\left(\ttpipi{1}, \ttpipi{2} \right)$.
	\EndIf
\EndFor
\State $\tilde\tau^\alpha\gets $ the $\lceil (k+1)(1-\alpha)\rceil$ ranked statistic in $q$.  
\If{$\tilde q > \tilde\tau^2 \qquad $}
 \Return Reject $H_0$.
\Else$\quad$ \Return Fail to Reject $H_0$.
\EndIf
\EndIf
\EndProcedure
\fi
\end{algorithmic}
\end{algorithm}

  \ifnum\submission=1
\begin{figure*}
\begin{center}
\includegraphics[width=.83\textwidth]{figures/GOFSignif_less}
\caption{Significance of test $\PG$ in $10,000$ trials with $(\epsilon,\delta) = (0.1,10^{-6})$.\label{fig:GOF_sig}}
\end{center}
\end{figure*}

\fi

\subsection{Monte Carlo Test: $\MCIND$}
We first follow a similar procedure as in \Cref{sec:MC} but using the parameter estimates from $\MLE$ instead of the actual (unknown) probabilities.  Our procedure $\MCIND$ (given in \Cref{alg:MC_ind} ) works as follows: given a dataset $\bbx$, we will add the appropriately scaled Laplace or Gaussian noise to ensure differential privacy to get the noisy table $\mathbf{w}$.  Then we use $\MLE$ on the private data to get approximates to the parameters $\pipi{i}$, which we denote as $\tpipi{i}$ for $i = 1,2$.  Using these probability estimates, we sample $k>1/\alpha$ many contingency tables and noise terms to get $k$ different values for $\tilde Q^2_\cD$ and choose the $\lceil (k+1)(1-\alpha)\rceil$ ranked statistic as our threshold $\tilde\tau^\alpha$.  If at any stage $\MLE$ returns $\NULL$, then the test Fails to Reject $H_0$.  We formally give our test $\MCIND$ in \Cref{alg:MC_ind}. 

%\begin{theorem}
%Assuming $n$ is known, $\MCIND(\cdot; (\epsilon,\delta),1-\alpha)$ is $\epsilon$-differentially private when using Laplace noise and $(\epsilon,\delta)$-differentially private when using Gaussian noise.
%\end{theorem}

\subsection{Asymptotic Approach: $\PI$}

\begin{algorithm}
\caption{ Private Independence Test for $r \times c$ tables}
\label{alg:Priv_Ind}
\begin{algorithmic}
\ifnum\submission=1
\REQUIRE $\PI(\bbx,(\epsilon,\delta),1-\alpha)$
\STATE $\mathbf{w} \gets \bbx + \bZ$ where $\bZ \sim N(0,\sigma^2I_{r \cdot c})$.
\STATE $\left(\tpipi{1},\tpipi{2}\right) \gets \MLE(\mathbf{w})$.
\IF{$\left(\tpipi{1},\tpipi{2}\right) == \NULL$}
\STATE $\return $ Fail to Reject
\ELSE 
\STATE $\tilde\bp \gets \bbf\left( \tpipi{1},\tpipi{2} \right)$ for $\bbf$ given in \eqref{eq:f}.
\STATE Set $\tilde Q^2_{Gauss}$ from \eqref{eq:hatQ_priv} with  $\mathbf{w}$ and $\tilde\tau^\alpha$ from \eqref{eq:ind_thresh}.
\LINEIF{$\tilde Q^2_{Gauss} > \tilde\tau^\alpha$}{$\quad \return$  Reject}
\LINEELSE{ $\quad \return$  Fail to Reject}
\ENDIF
\else
\Procedure{$\PI$}{Data $\bbx$, privacy parameters $(\epsilon,\delta)$, and Significance $1-\alpha$}
\State Compute the private contingency table $\bX + \bZ = \mathbf{w}$ where $\bZ \sim N(0,\sigma^2I_{r \cdot c})$ and $\sigma$ from \eqref{eq:PD}.
\State $\left(\tpipi{1},\tpipi{2}\right) \gets \MLE(\mathbf{w})$.
\If{$\left(\tpipi{1},\tpipi{2}\right) == \NULL$}
\State Decision $\gets$ Fail to Reject
\Else
\State $\tilde\bp \gets \bbf\left( \tpipi{1},\tpipi{2} \right)$ for $\bbf$ given in \eqref{eq:f}.
\State Compute $\tilde Q^2_{Gauss}$ from \eqref{eq:hatQ_priv} with noisy data $\mathbf{w}$.
\State Compute $\tilde\tau^\alpha$ that satisfies \eqref{eq:ind_thresh}.
\If{$\tilde Q^2_{Gauss} > \tilde\tau^\alpha$}
\State Decision $\gets$ Reject
\Else
\State Decision $\gets $ Fail to Reject
\EndIf
\EndIf
\State {\bf return} Decision.
\EndProcedure
\fi
\end{algorithmic}
\end{algorithm}

We will now focus on the analytical form of our private statistic when Guassian noise is added.  We can then write $\tilde Q^2_{Gauss}$  \ifnum\submission=1$= \mathbf{\tilde W}^T \tilde A \mathbf{\tilde W}$ \fi in its quadratic form, which is similar to the form of $Q^2_{Gauss}$ \ifnum\submission=0 from \eqref{eq:chisq_quad},\fi
\ifnum\submission=0
\begin{equation}
\tilde Q_{Gauss}^2 = \mathbf{\tilde W}^T \tilde A \mathbf{\tilde W}
\label{eq:chisq_ind_quad}
\end{equation}
\fi
where $\mathbf{\tilde W} = {\mathbf{\tilde U} \choose \mathbf{V}}$ with
$\mathbf{\tilde U}$ set in \eqref{eq:hatU} except with $\tilde\bp$ used instead of the given $\bp^0$ in the goodness of fit testing\ifnum\submission=0 \hspace{1mm} and $\mathbf{V}$ set as in \eqref{eq:W}\fi.  Further, we denote $\tilde A$ as $A$ in \eqref{eq:A} but with $\tilde\bp$ instead of $\bp^0$.  We will use the $2rc$ by $2rc$ block matrix $\tilde \Sigma'_{ind}$ to estimate the covariance of $\mathbf{\tilde W}$, 
 where
\begin{equation}
\tilde \Sigma_{ind}' = \begin{bmatrix}
			\tilde \Sigma_{ind} & 0 \\
			0 & I_{rc} 
			\end{bmatrix}
\end{equation}
and $\tilde \Sigma_{ind}$ is the matrix $ \Sigma_{ind}$ in
\Cref{lem:MLE_dist}, except we use our estimates $\tpipi{1}$, $\tpipi{2}$, or $\tilde\bp$ whenever we need to use the actual (unknown) parameters.  
%\fi

Thus, if we are given a differentially private version of a contingency table where each cell has added independent Gaussian noise with variance $\sigma^2$, we calculate $\tilde Q^2_{Gauss}$ and compare it to the threshold $\tilde\tau^\alpha$ where
\begin{equation}
\Prob{}{\sum_{i=1}^{rc} \tilde \lambda_i \chi^{2,i}_1 \geq \tilde\tau^\alpha} = \alpha
\label{eq:ind_thresh}
\end{equation}
with $\{ \tilde \lambda_i\}$ being the eigenvalues of $\tilde B ^T
\tilde A \tilde B$ with rank $\nu = rc +(r-1)(c-1)$ matrix $\tilde B \in \R^{2rc,\nu}$
where $\tilde B \tilde B^T = \tilde \Sigma_{ind}'$.  
%Note that the threshold changes with each realization of data and noise.
Our new independence test $\PI$ is given in
\Cref{alg:Priv_Ind}, where $\MLE$  estimates $\pipi{i}$ for $i = 1,2$ and $\PI$ Fails to Reject if $\MLE$ returns $\NULL$.%, from \Cref{alg:2MLE}.
%\begin{theorem}
%$\PI(\cdot; (\epsilon,\delta),1-\alpha,\bp^0)$ is $(\epsilon,\delta)$-differentially private assuming $n$ is known.
%\end{theorem}

\section{Significance Results}\label{sec:results}

\ifnum\submission=0
\begin{figure*}[h]
\begin{center}
\includegraphics[width=\textwidth]{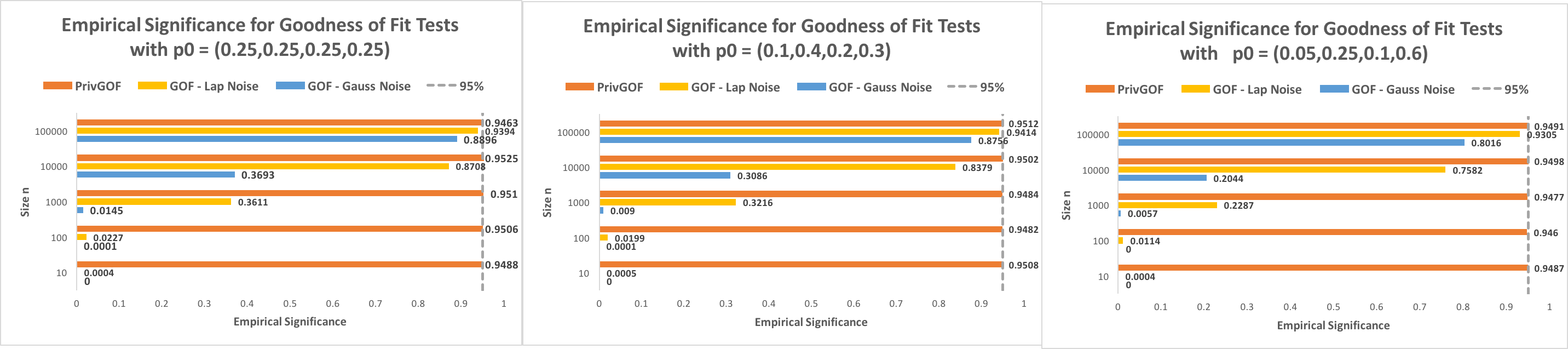}
\caption{Significance of the classical test $\GOF$ when used on counts with added Laplace or Gaussian noise compared to $\PG$ in $10,000$ trials with $(\epsilon,\delta) = (0.1,10^{-6})$ and $\alpha = 0.05$.\label{fig:GOF_sig}}
\end{center}
\end{figure*}

\begin{figure*}[h]
\begin{center}
\includegraphics[width=\textwidth]{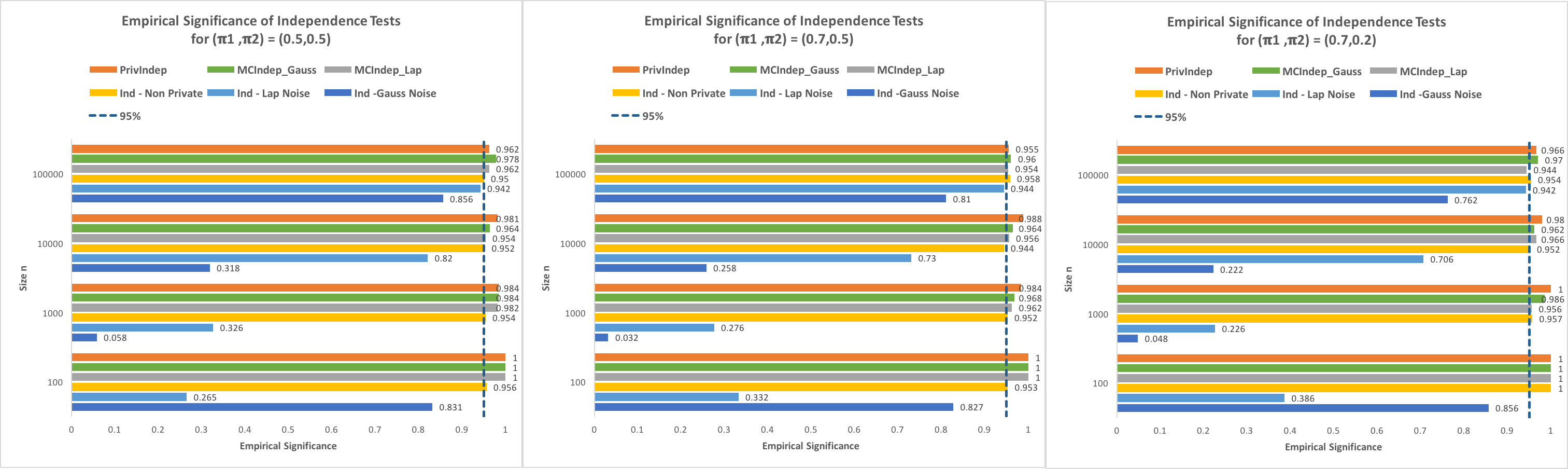}%more or less
\caption{Significance of $\IND$ when used on a contingency table with added Laplace or Gaussian noise compared to $\MCIND$ for both Laplace and Gaussian noise and $\PI$ in 1,000 trials with $(\epsilon,\delta) =
(0.1, 10^{-6})$ and $\alpha = 0.05$.\label{fig:IND_sig}}
\end{center}
\end{figure*}

\begin{figure*}[h]
\begin{center}
\includegraphics[width=\textwidth]{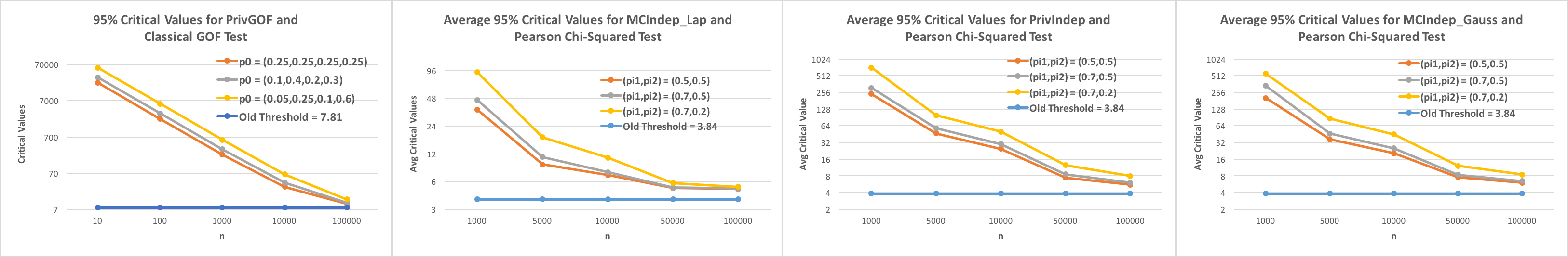}
\end{center}
\caption{\rynote{Modified $y$-axis to say ``Critical Value" instead of threshold}Comparison of the (average) critical values for all of our tests with $\alpha = 0.05$ and $(\epsilon,\delta) = (0.1,10^{-6})$.  Note that some are on a log-scale.\label{fig:thresholds}}
\end{figure*}

\fi

We now show how each of our tests perform on simulated data when $H_0$ holds in goodness of fit and independence testing.  We fix our desired significance $1-\alpha = 0.95$ and privacy level $(\epsilon,\delta) = (0.1,10^{-6})$ in all of our tests.  

By \Cref{thm:mcgof}, we know that $\MC_\cD$ will have significance at least $1-\alpha$.  We then turn to our test $\PG$ to compute the proportion of trials that failed to reject $H_0: \bp = \bp^0$ when it holds.  In \Cref{fig:GOF_sig} we give several different null hypotheses $\bp^0$ and sample sizes $n$ to show that $\PG$ achieves near $0.95$ significance in all our tested cases.  We also compare our results with how the original test $\GOF$ would perform if used on the private counts with either Laplace and Gaussian noise.

\ifnum\submission=0
To show that $\PG$ works beyond $d=4$ Multinomial data, we give a table of results in \Cref{tab:hresult_highd} for $d=100$ data and null hypothesis $p_i^0 = 1/100$ for $i\in [100]$.  We give the proportion of $10,000$ trials that were not rejected by $\PG$ in the ``$\PG$ Signf" column and those that were not rejected by the classical test $\GOF$ in the ``$\IND$ Signf" column.  Note that the critical value that $\GOF$ uses is $123.23$ for every test in this case, whereas $\PG$'s critical value changes for each test.  
\begin{table}[h]
\caption{Goodness of fit testing for multinomial data with $\alpha = 0.05$ and $(\epsilon,\delta) = (0.1,10^{-6})$ for dimension $d=100$ data.  } %title of the table
\label{table:d100}
\centering % centering table
\begin{tabular}{ rrr |c |c| c | c| c } 
\hline %inserting double-line
 \multicolumn{3}{c}{$\mathbf{p}^0$} \vline & $n$ & $\chi^2_{d-1,1-\alpha}$ & $\IND$ Signf & $\tau^\alpha$& $\PG$ Signf \\ [0.5ex]
\hline % inserts single-line
0.01 & $\cdots$ & 0.01 & 1,500 & 123.23 & 0.0000 & 48,231 & $0.9522$\\ % Entering row contents
\hline % inserts single-line
0.01 & $\cdots$ & 0.01 & 10,000 & 123.23 & 0.0000 & 7,339 & $0.9491$\\ % Entering row contents
\hline % inserts single-line
0.01 & $\cdots$ & 0.01 & 100,000 & 123.23 & 0.0000 & 844.7 & $0.9511$\\ % Entering row contents
\hline %
0.01 & $\cdots$ & 0.01 & 1,000,000 & 123.23 & 0.0524 & 195.3 & $0.9479$\\ % Entering row contents
\hline %
\end{tabular}
\label{tab:hresult_highd}
\end{table}
\fi

We then turn to independence testing for $2\times 2$ contingency tables using both $\MCIND$ and $\PI$. Note that our methods do apply to arbitrary $k \times \ell$ tables and run in time $\poly(k,\ell,\log(n))$ plus the time for the iterative Imhof method to find the critical values. In \Cref{fig:IND_sig} we compute the empirical significance of both of our tests and compare it to how $\IND$ performs on the nonprivate data.  For $\MCIND$ and $\PI$ we sample 1,000 trials for various parameters $\pipi{1}$, $\pipi{2}$, and $n$ that could have generated the contingency tables.  We set the number of samples $k = 50$ in $\MCIND$ regardless of the noise we added and when we use Laplace noise, we set $\gamma =0.01$ as the parameter in $\MLE$.  \ifnum\submission=0 Note that when $n$ is small, we get that our differentially private independence tests almost always fails to reject.  In fact, when $n=100$ all of our tests in 1,000 trials fail to reject.  This is due to $\MLE$ releasing a contingency table based on the private counts with small cell counts.  When the cell counts in $\MLE$ are small we follow the ``rule of thumb" from the classical test $\IND$ and output $\NULL$, which results in $\PI$ failing to reject. This will ensure good significance but makes no promises on power for small $n$, as does the classical test $\IND$.  Further, another consequence of this ``rule of thumb" is that when we use $\IND$ on private counts, with either Laplace or Gaussian noise, it tends to have lower Type I error than for larger $n$.  \fi

\ifnum\submission=1
\begin{figure*}
\begin{center}
\includegraphics[width=.75\textwidth]{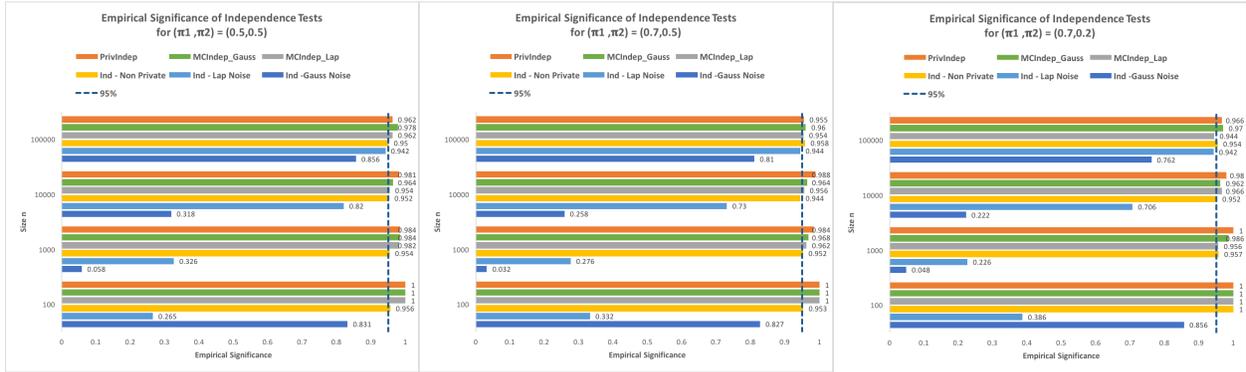}
\caption{Significance of $\IND$ when used on a contingency table with added Laplace or Gaussian noise compared to $\MCIND$ for both Laplace and Gaussian noise and $\PI$ in 1,000 trials with $(\epsilon,\delta) =
(0.1, 10^{-6})$ and $\alpha = 0.05$.\label{fig:IND_sig}}
\end{center}
\end{figure*}
\fi

\ifnum\submission=1
 \begin{figure*}
\begin{center}
\includegraphics[width=\textwidth]{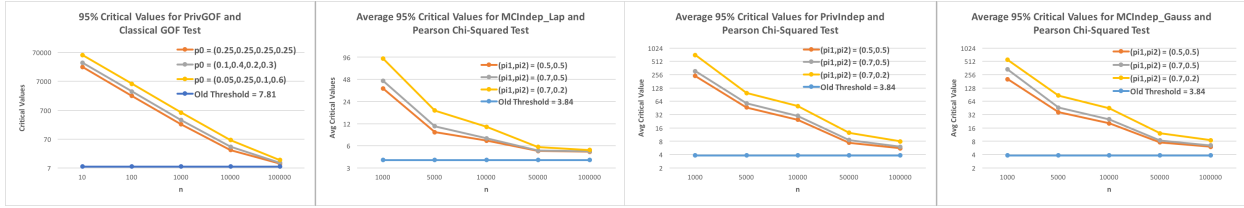}
\end{center}
\caption{Comparison of the (average) critical values for all of our tests with $\alpha = 0.05$ and $(\epsilon,\delta) = (0.1,10^{-6})$.\label{fig:thresholds}}
\end{figure*}
\fi

We also plot the critical values of our various tests in \Cref{fig:thresholds}.  For both $\PG$ and $\PI$ we used the package in R ``CompQuadForm" that has various methods for finding estimates to the tail probabilities for quadratic forms of normals, of which we used the ``imhof" method \citep{imhof} to approximate the threshold for each test.  Note that in $\MCIND$ and $\PI$ each trial has a different threshold, so we give the average over all trials.  

\section{Power Results}

We now want to show that our tests correctly reject $H_0$ when it is false, fixing parameters $\alpha = 0.05$ and $(\epsilon,\delta) = (0.1,10^{-6})$.  For our two goodness of fit tests, $\MC_\cD$ (with $k = 100$) and $\PG$ we test whether the multinomial data came from $\bp^0 = (1/4,1/4,1/4,1/4)$ when it was actually sampled from $\bp^1 = \bp^0+0.01\cdot(1,-1,1,-1)$.  We compare each of our tests with the classical $\IND$ test that uses the unaltered data in \Cref{fig:PowerPlots}.  We then find the proportion of 1,000 trials that each of our tests rejected $H_0: \bp= \bp^0$ for various $n$.  Note that $\IND$ has difficulty distinguishing $\bp^0$ and $\bp^1$ for reasonable sample sizes.

\begin{figure}[h]
\begin{center}
\ifnum\submission=1
\includegraphics[width=.5\textwidth]{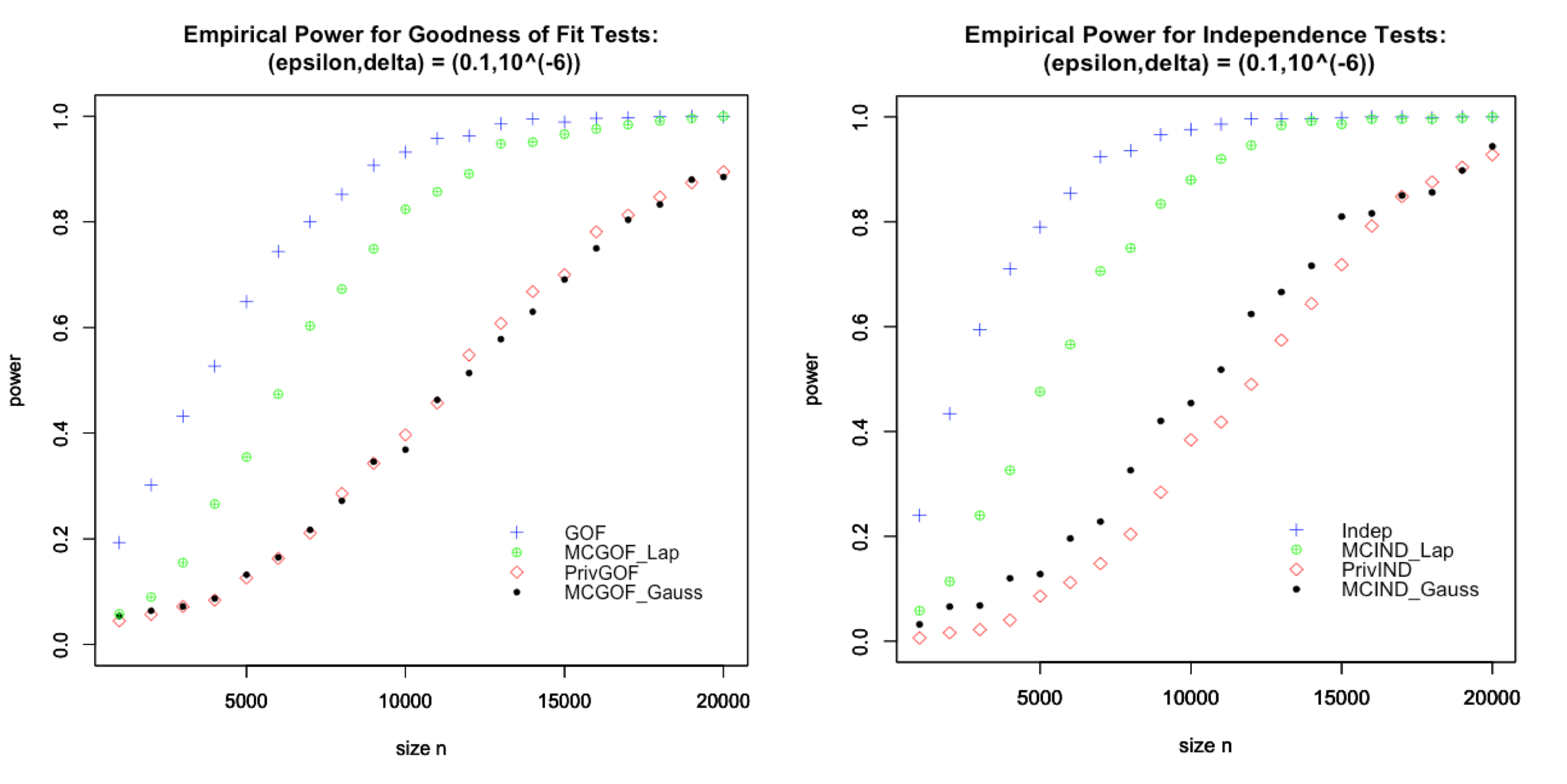}
\else
\includegraphics[width=\textwidth]{PowerPlots1}
\fi
\end{center}
\caption{Power Plots of $\MC_\cD$ and $\PG$\ifnum\submission=0 with alternate $\bp^1$ with parameter $\Delta = 0.01$\fi, as well as our independence tests $\MCIND$ and $\PI$ compared with the classical tests\ifnum\submission=0 with alternate covariance $0.01$\fi, with $(\epsilon,\delta) = (0.1,10^{-6})$.}
\label{fig:PowerPlots}
\end{figure}

We then turn to independence testing for $2 \times 2$ tables with our two differentially private tests $\MCIND$ and $\PI$.  We fix the alternate $H_1: \Cov(\YY{1},\YY{2}) = \Delta>0$ so that $\YY{1} \sim \Bern(\pipi{1} = 1/2)$ and $\YY{2} \sim \Bern(\pipi{2}=1/2)$ are not independent.  We then sample contingency tables from a multinomial distribution with probability $\bp^1 = (1/4,1/4,1/4,1/4)+\Delta (1,-1,1,-1)$ and various sizes $n$.  We compute the proportion of 1,000 trials that $\MCIND$ and $\PI$ rejected $H_0:\YY{1}\bot \YY{2}$ and $\Delta  = 0.01$ in \Cref{fig:PowerPlots}.  For $\MCIND$ we set the number of samples $k = 50$ and when we use Laplace noise, we set $\gamma = 0.01$ in $\MLE$.

\section{Conclusion}
We proposed new hypothesis tests based on a private version of the chi-squared statistic for goodness of fit and independence tests. For each test, we showed analytically or experimentally that we can achieve significance close to the target $1-\alpha$ level similar to the nonprivate tests. We also showed that all the tests have a loss in power with respect to the non-private classical tests, with methods using Laplace noise outperforming those with Gaussian noise, due to the fact that the Gaussian noise has higher variance (to achieve the same level of privacy).  Experimentally we show for $2 \times 2$ tables that with less than 3000 additional samples the tests with Laplace noise achieve the same power as the classical tests.  Typically, one would expect differential privacy to require the sample size to blow up by a multiplicative $1/\epsilon$ factor.  However, we see a better performance because the noise is dominated by the sampling error.  
 
%\section{Conclusion}
%\rynote{Add interpretation of results.}
%In this work we proposed two new hypothesis tests useful to
%reject or not a given model of the population based on sampled data
%and to preserve at the same time the privacy of the sampled data under
%the notion of differential privacy.  The two new tests are private
%versions of two classical tests based on the chi-squared statistic:
%the goodness of fit test and the test for independence.
%
%We showed that when Gaussian noise is added to each count of a
%multinomial distribution we can find a close approximation for the
%distribution of the chi-squared statistic.  Knowing this distribution
%allows researchers to compute p-values based on the differentially
%private statistic, ensuring the privacy of the sample data and
%allowing them to take decisions about the model.  
%
%We also showed that the new tests achieve significance levels that are
%similar to the ones achieved using the non-private tests. This ensures
%that the probability of rejecting the null hypothesis when it is true
%using the private tests is close to the one obtained by using the
%non-private test.  Moreover, we showed that for large sample size the
%two tests also achieve power results comparable to the non-private
%ones. This ensures that for large sample size the probability that the
%private tests rejects the null hypothesis when the alternate
%hypothesis is true is close to the one obtained using the non-private
%tests.  Overall, our results show that differentially private
%hypothesis testing is a new effective statistical tool.
\ifnum\submission=0
\section*{Acknowledgements}
We would like to thank the following people for helpful discussions: Dan Kifer, Aaron Roth, Aleksandra B. Slavkovi\'c, Or Sheffet, Adam Smith, and a number of others involved with the Privacy Tools for Sharing Research Data project.  A special thanks to Vishesh Karwa for helping us with statistics background and the suggestion to use a two step MLE procedure.  
\bibliography{refs}
\fi

\bibliographystyle{plainnat}
%\ifnum\submission=0
\clearpage
%\fi
\ifnum\submission=1
\section*{Acknowledgements}
This work is part of the ``Privacy Tools for Sharing Research Data'' project based at Harvard, supported by NSF grant CNS-1237235 as well as a grant from the Sloan Foundation.  This work has also been partially supported by the ``PrivInfer - Programming Languages for Differential Privacy: Conditioning and Inference'' EPSRC project EP/M022358/1 and by the University of Dundee, UK.  Salil Vadhan's work was also supported by a Simons Investigator grant and was done in part while visiting the Department of Applied Mathematics and the Shing-Tung Yau Center at National Chiao-Tung University in Taiwan.  We would like to thank the following people for helpful discussions: Dan Kifer, Aaron Roth, Aleksandra B. Slavkovi\'c, Or Sheffet, Adam Smith, and a number of others involved with the Privacy Tools for Sharing Research Data project.  A special thanks to Vishesh Karwa for helping us with statistics background and the suggestion to use a two step MLE procedure.  
\bibliography{refs}
\fi
%
%\ifnum\submission=0
%\appendix
%\input{appendix}
%\fi
\end{document}